\documentclass[leqno,10pt]{amsart}
\usepackage{amssymb,amsthm,wasysym}
\usepackage[latin2]{inputenc}

\theoremstyle{plain}

\allowdisplaybreaks

\newtheorem{thm}{Theorem}[section]

\newtheorem{pro}[thm]{Proposition}

\newtheorem{lem}[thm]{Lemma}
\newtheorem{rem}[thm]{Remark}
\newtheorem*{ex*}{Example}

\numberwithin{equation}{section}

\newcommand{\B}{\mathbb{B}}
\newcommand{\Z}{\mathbb{Z}_{2}^{d}}
\newcommand{\N}{\mathbb{N}}
\newcommand{\R}{\mathbb{R}_{+}^{d}}
\newcommand{\RR}{\mathbb{R}^{d}}
\def\a{\alpha}
\def\z{\zeta}
\def\ll{\lambda}
\def\t{\theta}
\def\v{\varphi}
\def\eps{\varepsilon}
\def\b{(\z,q_{\pm})}
\def\c{(\z,q_{\pm}(\t,y,s))}

\def\eee{(\z,q_{\pm}(x^{*},y,s))}
\def\fff{(\z,q_{\pm}(x,y^{*},s))}
\def\zz{(\z,q_{\pm}(x+z,y,s))}
\def\ww{(\z,q_{\pm}(\t+z,y,s))}

\DeclareMathOperator{\e}{{\mathbb{E}\textrm{xp}}}
\DeclareMathOperator{\lo}{{\mathbb{L}\textrm{og}}}
\DeclareMathOperator{\domain}{Dom}
\DeclareMathOperator{\supp}{supp}

\begin{document}

\subjclass[2000]{42C10 (primary), 42B25, 42B20 (secondary)}
\keywords{square function, $g$-function, Lusin's area integral, Dunkl's harmonic oscillator, generalized Hermite expansions, Laguerre semigroup, Laguerre expansions of convolution type, Calder\'on-Zygmund operator, $A_p$ weight}


\title[Square functions in certain Dunkl and Laguerre settings]{On Lusin's area integrals and g-functions in certain Dunkl and Laguerre settings}
\author[]{Tomasz Szarek}
\address{Tomasz Szarek,     \newline
      ul{.} W{.} Rutkiewicz 29\slash 43,
      PL-50--571 Wroc\l{}aw, Poland        \vspace{10pt}}

\email{szarektomaszz@gmail.com}

\begin{abstract}
We investigate $g$-functions and Lusin's area type integrals related to certain multi-dimen\-sional Dunkl and Laguerre settings. We prove that the considered square functions are bounded on weighted $L^p$, $1<p<\infty$, and from $L^1$ into weak $L^1$.
\end{abstract}

\maketitle

\section{Introduction}\label{intro}
\setcounter{equation}{0}

This paper embraces a completion and extension of the research initiated by the author in \cite{Sz} that concerned square functions related to the so-called Laguerre expansions of convolution type. Here we generalize the results of \cite{Sz} by studying square functions in the context of the Dunkl harmonic oscillator and the related group of reflections isomorphic to $\Z$. This Dunkl setting reduces to that of \cite{Sz} after restricting to reflection invariant functions. Consequently, the results delivered by the present paper implicitly contain, in particular, those of \cite{Sz}. Moreover, a trivial choice of the multiplicity function reduces the Dunkl setting to the situation of classical Hermite function expansions. Thus our results may also be seen as a continuation and extension of the investigations of Thangavelu \cite{T}, Harboure, de Rosa, Segovia and Torrea \cite{HRST} and Stempak and Torrea \cite{StTo2}, concerning $g$-functions in the context of the classic harmonic oscillator.  

An essential novelty in comparison with the previous study is the investigation of Lusin's area type integrals. These objects have more complex structure than the vertical and horizontal $g$-functions and hence their treatment requires additional arguments and effort. The results obtained in the Dunkl setting imply similar results in the Hermite setting and in the Laguerre situation of \cite{Sz}, where Lusin's area type integrals were not considered.

It is commonly known that square functions play an important role in harmonic analysis (see \cite[Section 1]{Sz} for brief comments and references), being valuable tools with several significant applications. Also the results we prove have some interesting potential applications, which remain to be investigated; this concerns, in particular, multiplier theorems and characterizations of Hardy spaces. Similarly to \cite{Sz}, the present work contributes to the development of Littlewood-Paley theory for discrete and continuous orthogonal expansions, which receives a considerable attention in recent years, see \cite[Section 1]{Sz} for references. In particular, Lusin's area type integrals in the context of another, one-dimensional, Laguerre setting, and also in the one-dimensional Hermite context, were studied very recently by Betancor, Molina and Rodr\'{\i}guez-Mesa \cite{BMR}. 

We refer the reader to the survey article by R\"osler \cite{R2} for basic facts concerning Dunkl's theory. A precise description of the Dunkl framework for the particular group of reflections $G$ isomorphic to $\Z$ can be found for instance in \cite[Section 3]{NS3}. Here we only invoke the most relevant facts. We shall work on the space $\RR$, $d\geq 1$, equipped with the measure
\begin{align*}
dw_{\a}(x)=\prod_{j=1}^{d}|x_{j}|^{2\a_{j}+1}\,dx,\qquad x=(x_1,\ldots,x_d)\in\RR,
\end{align*}
and with the Euclidean norm $|\cdot|$. The multi-index $\a=(\a_1,\ldots,\a_d)$ will always be assumed to belong to $[-1\slash 2,\infty)^d$. Consider the reflection group $G$ generated by $\sigma_{j}$, $j=1,\ldots,d$,
$$
\sigma_{j}(x_{1},\ldots,x_{j},\ldots,x_{d})=(x_{1},\ldots,-x_{j},\ldots,x_{d}).
$$
Clearly, the reflection $\sigma_{j}$ is in the hyperplane orthogonal to $e_j$, the $j$th coordinate vector. Notice that the measure $w_{\a}$ is $G$-invariant. The Dunkl differential-difference operators $T_{j}^{\a}$, $j=1,\ldots,d$, are given by
\begin{align*}
T_{j}^{\a}f(x)=\partial_{x_{j}}f(x)+(\a_{j}+1\slash 2)\frac{f(x)-f(\sigma_{j}x)}{x_{j}},\qquad f\in C^{1}(\mathbb{R}^{d}),\qquad j=1,\ldots,d,
\end{align*}
and form a commuting system. The Dunkl Laplacian,
\begin{align*}
\Delta_{\a}f(x)=\sum_{j=1}^{d}\big(T_{j}^{\a}\big)^{2}f(x)
=
\sum_{j=1}^{d}\bigg(\frac{\partial^{2}f}{\partial x_j^2}(x)+\frac{2\a_j+1}{x_j}\frac{\partial f}{\partial x_j}(x)-(\a_j+1\slash 2)\frac{f(x)-f(\sigma_j x)}{x_j^2}\bigg),
\end{align*} 
is formally self-adjoint in $L^2(\RR,dw_{\a})$. The Dunkl harmonic oscillator is defined as
\begin{align*}
L_{\a}=-\Delta_{\a}+|x|^{2}.
\end{align*}
This operator will play in the present paper a similar role to that of the Euclidean Laplacian in the classical harmonic analysis. Note that for $\a=(-1\slash 2,\ldots,-1\slash 2)$, $L_\a$ becomes the classic harmonic oscillator $-\Delta+|x|^2$. We shall consider a self-adjoint extension $\mathcal{L}_\a$ of $L_\a$, whose spectral decomposition is discrete and given by the generalized Hermite functions $h_n^\a$, see Section \ref{pre} for details. Natural partial derivatives related to $L_\a$ are obtained from the symmetric decomposition 
\begin{align*}
L_{\a}=\frac{1}{2}\sum_{j=1}^{d}(\delta_{j}^{*}\delta_{j}+\delta_{j}\delta_{j}^{*}),
\end{align*}
where
\begin{align*}
\delta_{j}=T_{j}^{\a}+x_{j},\qquad \delta_{j}^{*}=-T_{j}^{\a}+x_{j},\qquad j=1,\ldots,d;
\end{align*}
here $\delta_j^*$ is the formal adjoint of $\delta_j$ in $L^2(\RR,dw_\a)$.

The main objects of our study are vertical and horizontal $g$-functions and Lusin's type area integrals based on the semigroup generated by $\mathcal{L}_\a$. Our main result, Theorem \ref{main} below, says that each of the square functions is bounded on weighted $L^p(dw_\a)$, $1<p<\infty$, and satisfies weighted weak type (1,1) inequality for a large class of weights. To prove this, we exploit the arguments from \cite{NS2} that allow to reduce the analysis to the context of the smaller measure space $(\R,dw_{\a}^{+})$ and suitably defined Laguerre-type square functions, where $\R=(0,\infty)^d$ and $w_\a^+$ is the restriction of $w_\a$ to $\R$. Then we apply the general theory of vector-valued Calder\'on-Zygmund operators with the underlying space of homogeneous type $(\R,dw_{\a}^{+},|\cdot|)$. The main technical difficulty connected with this approach is to show the relevant kernel estimates. Here, similarly as in \cite{Sz}, we use a convenient technique having roots in Sasso's work \cite{Sa} and developed later by Nowak and Stempak in \cite{NS1}. For our purposes we derive some further generalizations of this interesting method, which may be of independent interest. It is remarkable that essentially the same procedure applies as well to higher order square functions in the investigated setting. The related analysis, however, is because of its length beyond the scope of this article.

The paper is organized as follows. Section \ref{pre} contains the setup, definitions of the investigated square functions, statements of the main results and the accompanying comments and remarks. Also, suitable Laguerre-type square functions, related to the restricted space $(\R,dw_{\a}^{+})$, are defined and the proof of the main theorem is reduced to showing that these auxiliary square functions can be viewed as vector-valued Calder\'on-Zygmund operators. In Section \ref{o+s} the Laguerre-type square functions are proved to be $L^2$-bounded and associated, in the Calder\'on-Zygmund theory sense, with the relevant kernels. Finally, Section \ref{ker} is devoted to the proofs of all necessary kernel estimates. This is the largest and most technical part of the work.

Throughout the paper we use a standard notation with essentially all symbols referring to the spaces $(\RR,dw_{\a},|\cdot|)$ or $(\R,dw_{\a}^{+},|\cdot|)$. Thus $\Delta$ and $\nabla$ denote the Euclidean Laplacian and gradient, respectively. Further, $L^{p}(\RR,Wdw_{\a})$ stands for the weighted $L^{p}(\RR,dw_{\a})$ space, $W$ being a nonnegative weight on $\RR$; we write simply $L^{p}(dw_{\a})$ if $W\equiv 1$. 
By $\langle f,g\rangle_{dw_{\a}}$ we mean $\int_{\RR}f(x)\overline{g(x)}\,dw_{\a}(x)$ whenever the integral makes sense. In a similar way we define $L^{p}(\R,Wdw_{\a}^{+})$ and $\langle f,g\rangle_{dw_{\a}^+}$. For  $1\leq p<\infty$ we denote by $A_{p}^{\a,+}$ the Muckenhoupt class of $A_{p}$ weights associated to the space $(\R,dw_{\a}^+,|\cdot|)$. 

While writing estimates we will frequently use the notation $X\lesssim Y$ to indicate that $X\leq CY$ with a positive constant $C$ independent of significant quantities. We will write $X\simeq Y$ when $X\lesssim Y$ and $Y\lesssim X$.

\textbf{Acknowledgments.} The author would like to thank Dr. Adam Nowak for suggesting the topic and for many discussions related to this paper.

\section{Preliminaries and statement of results}\label{pre}
\setcounter{equation}{0}

Let $m=(m_1,\ldots,m_d) \in \N^d$, $\N=\{0,1,\dots\}$,  and
$\a = (\a _1, \ldots , \a _d) \in [-1\slash 2,\infty)^{d}$
be multi-indices. The generalized Hermite functions in $\RR$ are defined as the tensor products
$$
h_{m}^{\alpha}(x) = h _{m_1}^{\alpha _1}(x_1) \cdot \ldots \cdot
h_{m_d}^{\alpha _d}(x_d), \qquad x = (x_1, \ldots ,x_d)\in \RR,
$$
where $h_{m_i}^{\alpha _i}$ are the one-dimensional generalized Hermite functions
\begin{align*}
h_{2m_i}^{\alpha_i}(x_i)&=d_{2m_i,\alpha_i}e^{-x_i^2/2}L_{m_i}^{\alpha_i}(x_i^2),\\
h_{2m_i+1}^{\alpha _i}(x_i)&=d_{2m_i+1,\alpha_i}e^{-x_i^2/2}x_iL_{m_i}^{\alpha_i+1}(x_i^2);
\end{align*}
here $L_{m_i}^{\alpha_i}$ is the Laguerre polynomial of degree $m_{i}$ and order $\a_i$, and $d_{k,\a_{i}}$, $k\in\N$, are proper normalizing constants, see \cite[p.\,544]{NS3} or \cite[p.\,4]{NS2}. The system $\{h_{m}^{\a} : m\in\N^{d} \}$ is an orthonormal basis in $L^{2}(\RR,dw_{\a})$ consisting of eigenfunctions of $L_{\a}$,
\begin{align*}
L_{\a}h_{m}^{\a}=\ll_{|m|}^{\a} h_{m}^{\a},\qquad \ll_n^\a=2n+2|\a|+2d,\qquad n\in\N;
\end{align*}
here $|m|=m_1+\ldots+m_d$ is the length of $m$.
The operator
\begin{align*}
\mathcal{L}_\alpha f=\sum_{n=0}^\infty\ll_{n}^{\a}
\sum_{|m|=n}\langle f,h_m^{\alpha}
  \rangle_{dw_{\alpha}}h_m^{\alpha},
\end{align*}
defined on the domain
$$
  \domain(\mathcal{L}_{\a})=\Big\{f\in L^2(\RR, dw_{{\a}}):
  \sum_{m\in \N^{d}}\big|\ll_{|m|}^\a\,
\langle f,h_m^\a
  \rangle_{dw_{\a}}\big|^2<\infty\Big\},
$$
is a self-adjoint extension of $L_{\a}$ considered on $C^{\infty}_c(\RR)$ as the natural domain (the inclusion $C^{\infty}_c(\RR)\subset\domain(\mathcal{L}_{\a})$ may be easily verified).

The heat semigroup $T_t^\alpha = \exp(-t\mathcal{L}_{\alpha})$, $t\geq 0$, generated by $\mathcal{L}_{\alpha}$ is a strongly continuous semigroup of contractions on
$L^2(\RR, dw_{\alpha})$. By the spectral theorem,
\begin{equation*}
T_t^\alpha f=\sum_{n=0}^\infty e^{-t\ll_{n}^{\a}}\sum_{|m|=n}\langle f,h_m^{\alpha}
  \rangle_{dw_{\alpha}}h_m^{\alpha}, 
\qquad f\in L^2(\RR,dw_{\alpha}).
\end{equation*}
We have the integral representation
\begin{equation*}
  T_t^\alpha f(x)=\int_{\RR} G_t^\alpha(x,y)f(y)\,dw_{\alpha}(y),
  \qquad x\in \RR, \quad t>0,
\end{equation*}
where the Dunkl heat kernel is given by
\begin{equation}\label{gie}
G^\alpha_t(x,y)=\sum_{n=0}^\infty e^{-t\ll_{n}^{\a}} \sum_{|m|=n}
h_m^\alpha(x)h_m^\alpha(y).
\end{equation}
This oscillating series can be summed, see for instance \cite[p.\,544]{NS3} or \cite[p.\,5]{NS2}, and the resulting formula is
\begin{align*}
G^{\a}_{t}(x,y)=\sum_{\eps\in\Z}G^{\a,\eps}_{t}(x,y),
\end{align*}
with the component kernels
$$
G^{\a,\eps}_{t}(x,y)=(2\sinh 2t)^{-d}\exp\Big({-\frac{1}{2} \coth(2t)\big(|x|^{2}+|y|^{2}\big)}\Big)
\prod^{d}_{i=1} (x_i y_i)^{\eps_{i}} \frac{I_{\alpha_i+\eps_{i}}\left(\frac{x_i y_i}{\sinh 2t}\right)}{(x_i y_i)^{\a_{i}+\eps_{i}}},
$$
where $I_{\nu}$ denotes the modified Bessel function of the first kind and order $\nu$. Here we consider the functions $z\mapsto z^{\nu}$ and the Bessel function as analytic functions on $\mathbb{C}$ cut along the half axis $\{ix : x\leq 0\}$, see the references given above. Note that $G_{t}^{\a,\eps}(x,y)$ is also expressed by the series \eqref{gie}, but with the summation in $m$ restricted to the set
$$
\mathcal{N}_{\eps} = \big\{ m \in \N^{d} : m_i \; \textrm{is even if} \; \eps_i=0,
	\; m_i \; \textrm{is odd if}\; \eps_i=1,\; i=1,\ldots,d\big\}.
$$
The operators determined by integration against $G_t^{\a,\eps}(x,y)dw_{\a}(y)$, $\eps\in\Z$, will be denoted by $T_t^{\a,\eps}$. Clearly, we have the decomposition
\begin{align}\label{dec}
T_t^\a=\sum_{\eps\in\Z}T_t^{\a,\eps}.
\end{align}

We consider the following vertical and horizontal square functions based on the Dunkl heat semigroup:
\begin{displaymath}
\begin {array}{lll}
g_{V}(f)(x)&=
\big\|\partial_{t}T_{t}^{\a}f(x)\big\|_{L^{2}(tdt)},&\\
g_{H}^{j}(f)(x)&=
\big\|\delta_{j}T_{t}^{\a}f(x)\big\|_{L^{2}(dt)},\qquad&
j=1,\ldots,d,\\
g_{H,*}^{j}(f)(x)&=
\big\|\delta_{j}^{*}T_{t}^{\a}f(x)\big\|_{L^{2}(dt)},\qquad&
j=1,\ldots,d,\\
S_{V}(f)(x)&=
\Big(\int_{A(x)}t\big|\partial_{t}T_{t}^{\a}f(z)\big|^{2}\frac{dw_{\a}(z)}
{V_{\sqrt{t}}^{\a}(x)}\,dt\Big)^{1\slash 2},&\\
S_{H}^{j}(f)(x)&=
\Big(\int_{A(x)}\big|\delta_{j}T_{t}^{\a}f(z)\big|^{2}\frac{dw_{\a}(z)}
{V_{\sqrt{t}}^{\a}(x)}\,dt\Big)^{1\slash 2},\qquad&
j=1,\ldots,d,\\
S_{H,*}^{j}(f)(x)&=
\Big(\int_{A(x)}\big|\delta_{j}^{*}T_{t}^{\a}f(z)\big|^{2}\frac{dw_{\a}(z)}
{V_{\sqrt{t}}^{\a}(x)}\,dt\Big)^{1\slash 2},\qquad&
j=1,\ldots,d,
\end {array}
\end{displaymath}
where $A(x)$ is the parabolic cone with vertex at $x$,
$$
A(x)=(x,0)+A,\qquad 
A=\Big\{(z,t)\in \RR\times(0,\infty) : |z|<\sqrt{t}\Big\},
$$
and $V_{t}^{\a}(x)$ is the $w_{\a}$ measure of the cube centered at $x$ and with side lengths $2t$. More precisely,
\begin{align*}
V_{t}^{\a}(x)=\prod_{j=1}^{d}V_{t}^{\a_{j}}(x_j),\qquad 
V_{t}^{\a_j}(x_j)=w_{\a_{j}}\big((x_{j}-t,x_{j}+t)\big),\qquad x\in\RR,\quad t>0.
\end{align*}
The above definitions of $S_{V}$, $S_{H}^{j}$, $S_{H,*}^{j}$ fit into a general concept of Lusin's area integrals in a context of spaces of homogeneous type; see for instance \cite[(2.10)]{HLMMY} or \cite[Section 1]{BMR}.
It is not hard to see that the area type integrals just defined can be written as
\begin{displaymath}
\begin {array}{lll}
S_{V}(f)(x)&=
\Big\|\partial_{t}T_{t}^{\a}f(x+z)\sqrt{\frac{w_{\a}(x+z)}
{V_{\sqrt{t}}^{\a}(x)}}\Big\|_{L^2(A,tdtdz)},&\\
S_{H}^{j}(f)(x)&=
\Big\|\delta_{j}T_{t}^{\a}f(x+z)\sqrt{\frac{w_{\a}(x+z)}
{V_{\sqrt{t}}^{\a}(x)}}\Big\|_{L^2(A,dtdz)},\qquad&
j=1,\ldots,d,\\
S_{H,*}^{j}(f)(x)&=
\Big\|\delta_{j}^{*}T_{t}^{\a}f(x+z)\sqrt{\frac{w_{\a}(x+z)}
{V_{\sqrt{t}}^{\a}(x)}}\Big\|_{L^2(A,dtdz)},\qquad&
j=1,\ldots,d.
\end {array}
\end{displaymath}

Our main result concerns mapping properties of the square functions under consideration.

\begin{thm}\label{main}
Assume that $\a\in[-1\slash 2,\infty)^{d}$ and $W$ is a weight on $\RR$ invariant under the reflections $\sigma_1,\ldots,\sigma_d$. Then each of the square functions 
\begin{align*}
g_{V},\quad g_{H}^{j},\quad g_{H,*}^{j},\quad S_{V},\quad S_{H}^{j},\quad S_{H,*}^{j},\qquad j=1,\ldots,d,
\end{align*}
is bounded on $L^{p}(\RR,Wdw_{\a})$, $W^{+}\in A_{p}^{\a,+}$, $1<p<\infty$, and from $L^{1}(\RR,Wdw_{\a})$ to weak $L^{1}(\RR,Wdw_{\a})$, $W^{+}\in A_{1}^{\a,+}$.
\end{thm}

Proving Theorem \ref{main} can be reduced to showing similar mapping properties for certain square functions emerging from those defined above and related to the restricted space $(\R,dw_\a^+)$; recall that $w_{\a}^{+}$ is the restriction of $w_{\a}$ to $\R$. The details are as follows. For $\eps\in\Z$, we consider the operators acting on $L^2(\R,dw_{\a}^{+})$ and defined by
\begin{align*}
T_{t}^{\a,\eps,+}f=
\sum_{n=0}^{\infty}e^{-t\ll_{n}^{\a}}
\sum_{\substack{|m|=n\\ m\in \mathcal{N}_{\eps}}}
\langle f,h_{m}^{\a}\rangle_{dw_{\a}^{+}}h_{m}^{\a},\qquad f\in L^{2}(\R,dw_{\a}^{+}).
\end{align*}
The integral representation of $T_t^{\a,\eps,+}$ is
\begin{align*}
T_t^{\alpha,\eps,+} f(x)=\int_{\R} G_t^{\alpha,\eps}(x,y)f(y)\,dw_{\alpha}^{+}(y),
  \qquad x\in \R, \quad t>0.
\end{align*} 
The estimates
\begin{align}\label{oszh1}
|h_{m}^{\a}(x)|
&\lesssim 
(|m|+1)^{c_{d,\a}}
,\qquad m\in \N^{d},\quad x\in\R,\\\label{oszh2}
|\langle f,h_{m}^{\a}\rangle_{dw_{\a}^{+}}|
&\lesssim
\big(|m|+1\big)^{c_{d,\a,p}}\|f\|_{L^{p}(\R,Udw_{\a}^{+})},\qquad m\in \N^{d},
\end{align}
which hold for general $f\in L^p(\R,Udw_\a^+)$, $U\in A_{p}^{\a,+}$, $1\leq p<\infty$,
allow to check that for each $\eps\in\Z$ the series defining $T_{t}^{\a,\eps,+}$ converges pointwise for such $f$ and produces a smooth function of $(t,x)\in(0,\infty)\times\R$. An analogous claim is true for the integral representation.
The bound \eqref{oszh1} is a consequence of Muckenhoupt's generalization \cite{Mu} of the classical estimates for the standard Laguerre functions due to Askey and Wainger \cite{AW}. Actually, those estimates imply a sharper version of \eqref{oszh1} that involves some exponential decay in $x$, which together with the arguments from the proof of \cite[Lemma 4.2]{No} justifies \eqref{oszh2}.

Next, we define the Laguerre-type square functions
\begin{displaymath}
\begin {array}{lll}
g_{V}^{\eps,+}(f)(x)&=
\big\|\partial_{t}T_{t}^{\a,\eps,+}f(x)\big\|_{L^{2}(tdt)},&\\
g_{H}^{j,\eps,+}(f)(x)&=
\big\|\delta_{j}T_{t}^{\a,\eps,+}f(x)\big\|_{L^{2}(dt)},\qquad&
j=1,\ldots,d,\\
g_{H,*}^{j,\eps,+}(f)(x)&=
\big\|\delta_{j}^{*}T_{t}^{\a,\eps,+}f(x)\big\|_{L^{2}(dt)},\qquad&
j=1,\ldots,d,\\
S_{V}^{\eps,+}(f)(x)&=
\Big(\int_{A(x)}t\big|\partial_{t}T_{t}^{\a,\eps,+}f(z)\big|^{2}
\chi_{\{z\in\R\}}
\frac{dw_{\a}^{+}(z)}{V_{\sqrt{t}}^{\a,+}(x)}
\,dt\Big)^{1\slash 2},&\\
S_{H}^{j,\eps,+}(f)(x)&=
\Big(\int_{A(x)}\big|\delta_{j}T_{t}^{\a,\eps,+}f(z)\big|^{2}
\chi_{\{z\in\R\}}
\frac{dw_{\a}^{+}(z)}{V_{\sqrt{t}}^{\a,+}(x)}
\,dt\Big)^{1\slash 2},\qquad&
j=1,\ldots,d,\\
S_{H,*}^{j,\eps,+}(f)(x)&=
\Big(\int_{A(x)}\big|\delta_{j}^{*}T_{t}^{\a,\eps,+}f(z)\big|^{2}
\chi_{\{z\in\R\}}
\frac{dw_{\a}^{+}(z)}{V_{\sqrt{t}}^{\a,+}(x)}
\,dt\Big)^{1\slash 2},\qquad&
j=1,\ldots,d.
\end {array}
\end{displaymath}
Here $V_{t}^{\a,+}(x)$ denotes the $w_{\a}^{+}$ measure of the cube centered at $x$ and with side lengths $2t$, restricted to $\R$. More precisely,
\begin{align}\label{defV+}
V_{t}^{\a,+}(x)=\prod_{j=1}^{d}V_{t}^{\a_{j},+}(x_j),\qquad x\in\R,\quad t>0,
\end{align}
and for $j=1,\ldots,d$,
\begin{align*}
V_{t}^{\a_j,+}(x_j)
=
w_{\a_{j}}^{+}\big((x_{j}-t,x_{j}+t)\cap\mathbb{R}_{+}\big)
=
\left\{ \begin{array}{ll}
\frac{(x_{j}+t)^{2\a_j+2}}{2\a_j+2}, & x_j<t\\
\frac{(x_{j}+t)^{2\a_j+2}-(x_{j}-t)^{2\a_j+2}}{2\a_j+2}, & x_j\geq t
\end{array} \right..
\end{align*}
Notice that 
\begin{align}\label{esv}
V_{t}^{\a,+}(x)\simeq t^{d}\prod_{j=1}^{d}(x_{j}+t)^{2\a_{j}+1},
\qquad x\in\R,\quad t>0.
\end{align}
Observe also that the Laguerre-type Lusin's area integrals can be written as
\begin{displaymath}
\begin {array}{lll}
S_{V}^{\eps,+}(f)(x)&=
\big\|\partial_{t}T_{t}^{\a,\eps,+}f(x+z)\sqrt{\v_{\a}(x,z,t)}
\,\chi_{\{x+z\in\R\}}\big\|_{L^2(A,tdtdz)},&\\
S_{H}^{j,\eps,+}(f)(x)&=
\big\|\delta_{j}T_{t}^{\a,\eps,+}f(x+z)\sqrt{\v_{\a}(x,z,t)}
\,\chi_{\{x+z\in\R\}}\big\|_{L^2(A,dtdz)},\qquad&
j=1,\ldots,d,\\
S_{H,*}^{j,\eps,+}(f)(x)&=
\big\|\delta_{j}^{*}T_{t}^{\a,\eps,+}f(x+z)\sqrt{\v_{\a}(x,z,t)}
\,\chi_{\{x+z\in\R\}}\big\|_{L^2(A,dtdz)},\qquad&
j=1,\ldots,d,
\end {array}
\end{displaymath}
where the function $\v_{\a}$ is given by
\begin{equation}\label{deffi}
\v_{\a}(x,z,t)=\prod_{j=1}^{d}
\frac{(x_{j}+z_{j})^{2\a_{j}+1}}{V_{\sqrt{t}}^{\a_{j},+}(x_{j})},\qquad x\in\R,\quad z\in\RR,\quad x+z\in\R.
\end{equation}
We are now in a position to reduce the proof of Theorem \ref{main} to showing the following.

\begin{thm}\label{main+}
Assume that $\a\in[-1\slash 2,\infty)^{d}$ and $\eps\in\Z$. Then each of the Laguerre-type square functions 
\begin{align*}
g_{V}^{\eps,+},\quad g_{H}^{j,\eps,+},\quad g_{H,*}^{j,\eps,+},\quad
S_{V}^{\eps,+},\quad S_{H}^{j,\eps,+},\quad S_{H,*}^{j,\eps,+},\qquad
j=1,\ldots,d,
\end{align*}
is bounded on $L^{p}(\R,U dw_{\a}^{+})$, $U\in A_{p}^{\a,+}$, $1<p<\infty$, and from $L^1(\R,U dw_{\a}^{+})$ to weak $L^{1}(\R,U dw_{\a}^{+})$, $U\in A_{1}^{\a,+}$.
\end{thm}

For the sake of brevity, we give a detailed description of the reduction only in the case of $S_{H}^{j}$, adapting suitably arguments from the proof of \cite[Theorem 1]{NS2}. The remaining cases are treated in a similar way and the cases of $g_{V}$, $g_{H}^{j}$ and $g_{H,*}^{j}$ are even simpler.
In what follows, we shall use the following terminology. Given $\eps\in\Z$, we say that a function $f\colon\RR\to\mathbb{C}$ is $\eps$-symmetric if for each $j=1,\ldots,d$, $f$ is either even or odd with respect to the $j$th coordinate according to whether $\eps_{j}=0$ or $\eps_{j}=1$, respectively. If $f$ is $(0,\ldots,0)$-symmetric, then we simply say that $f$ is symmetric. Furthermore, if there exists $\eps\in\Z$ such that $f$ is $\eps$-symmetric, then we denote by $f^{+}$ the restriction of $f$ to $\R$. This convention pertains also to $\eps$-symmetric weights defined on $\RR$.

Let $j\in\{1,\ldots,d\}$ and $1\leq p<\infty$ be fixed, and let $W$ be a symmetric weight on $\RR$ such that $W^+\in A_p^{\a,+}$. According to \eqref{dec}, we decompose $\delta_{j}T_{t}^{\a}$ into a finite sum,
\begin{align*}
\delta_{j}T_{t}^{\a} f=
\sum_{\eps\in\Z}\delta_{j}T_{t}^{\a,\eps} f.
\end{align*}
Next, we invoke the differentiation rule (see \cite[(4.4)]{NS3})
\begin{align*}
\delta_j h_m^\a=\Phi(m_j,\a_j) h_{m-e_j}^\a,
\end{align*}
where
\begin{align*}
\Phi(m_{j},\a_{j})=
\left\{ \begin{array}{ll}
\sqrt{2m_{j}} & \textrm{if} \; m_{j} \; \textrm{is even}\\
\sqrt{2m_{j}+4\a_{j}+2} & \textrm{if} \; m_{j} \; \textrm{is odd}
\end{array} \right. ;
\end{align*}
here and elsewhere we use the convention that $h_{m}^{\a}=0$ if $m\notin \N^{d}$. Then, in view of the estimates similar to \eqref{oszh1} and \eqref{oszh2}, but adjusted to the space $(\RR,dw_\a,|\cdot|)$, we may write
\begin{align*}
\delta_{j}T_{t}^{\a,\eps} f=
\sum_{n=0}^{\infty}e^{-t\ll_{n}^{\a}}
\sum_{\substack{|m|=n\\ m\in \mathcal{N}_{\eps}}}
\langle f,h_{m}^{\a}\rangle_{dw_{\a}}
\Phi(m_{j},\a_{j})h_{m-e_{j}}^{\a},\qquad f\in L^p(\RR,Wdw_{\a}).
\end{align*}
Proceeding as in \cite[Section 3]{NS2}, we split a function $f\in L^{p}(\RR,Wdw_{\a})$ into a sum of $\eps$-symmetric functions $f_{\eps}$,
$$
f = \sum_{\varepsilon \in \mathbb{Z}_2^d} f_{\varepsilon}, \qquad \qquad
	f_{\varepsilon}(x) = \frac{1}{2^d} \sum_{\eta \in \{-1,1\}^d} \eta^{\varepsilon} f(\eta x),
$$
where $\eta^{\eps}=\eta_1^{\eps_1}\cdot\ldots\cdot\eta_d^{\eps_d}$ and $\eta x=(\eta_1 x_1,\ldots,\eta_d x_d)$.
Since $h_{m}^{\a}$ is $\eps$-symmetric if and only if $m\in\mathcal{N}_{\eps}$, we see that
\begin{align}\label{eqt}
\delta_{j}T_{t}^{\a} f
=\sum_{\eps\in\Z}\delta_{j}T_{t}^{\a,\eps} f
=\sum_{\eps\in\Z}\delta_{j}T_{t}^{\a,\eps} f_{\eps},
\end{align}
and the function $\delta_{j}T_{t}^{\a,\eps} f_{\eps}$ is $(\eps\pm e_{j})$-symmetric, depending on whether $\eps_{j}=0$ or $\eps_{j}=1$. 

Consider the auxiliary square functions $\mathcal{S}_{H}^{j,\eps}$, $\eps\in\Z$, acting on functions on $\RR$ and defined by
\begin{align*}
\mathcal{S}_{H}^{j,\eps} h(x)=
\bigg(\int_{A}\big|\delta_{j}T_{t}^{\a,\eps}h(x+z)\big|^{2}
\frac{w_{\a}(x+z)}{V_{\sqrt{t}}^{\a}(x)}
\,dz\,dt\bigg)^{1\slash 2}.
\end{align*}
Since $|\delta_{j}T_{t}^{\a,\eps} f_{\eps}|$ and $w_{\a}$ are symmetric, and $A$ is a symmetric set, it follows that $\mathcal{S}_{H}^{j,\eps}f_{\eps}$ is also symmetric. Moreover, by \eqref{eqt} we see that
\begin{align*}
S_{H}^{j} (f)(x)\leq\sum_{\eps\in\Z}\mathcal{S}_{H}^{j,\eps}f_{\eps}(x).
\end{align*}
Now, by the inclusions 
\begin{align*}
\big\{z\in\RR : |z-x|<\sqrt{t} \big\}\subset 
\Big(\bigcup_{\eta\in\Z}
\big\{z\in\mathbb{R}_{\eta} : |z-\sigma^{\eta}(x)|<\sqrt{t}\big\}\Big)
\cup M,\qquad x\in\R,\quad t>0, 
\end{align*}
where
\begin{align*}
M=&\big\{z\in\RR : \; \textrm{there exists} \; i\in\{1,\ldots,d\} \; \textrm{such that} \; z_{i}=0\big\},\\
\mathbb{R}_{\eta}=&\big\{z\in\RR : z_{i}>0 \; \textrm{if} \; \eta_i=0,
\; z_{i}<0 \; \textrm{if} \; \eta_i=1, \; i=1,\ldots,d \big\},
\end{align*}
and $\sigma^{\eta}=\sigma_{1}^{\eta_{1}}\circ\ldots\circ\sigma_{d}^{\eta_{d}}$, we get for any $x\in\R$,
\begin{align*}
\big(\mathcal{S}_{H}^{j,\eps} f_{\eps}(x)\big)^{2}&=
\int_{|z-x|<\sqrt{t}}\big|\delta_{j}T_{t}^{\a,\eps}f_{\eps}(z)\big|^{2}
\frac{w_{\a}(z)}{V_{\sqrt{t}}^{\a}(x)}\,dz\,dt\\
&\leq
\sum_{\eta\in\Z}
\int_{|z-\sigma^{\eta}(x)|<\sqrt{t}}\big|\delta_{j}T_{t}^{\a,\eps}f_{\eps}(z)\big|^{2}
\frac{w_{\a}(z)}{V_{\sqrt{t}}^{\a}(x)}\,\chi_{\{z\in\mathbb{R}_{\eta}\}}\,dz\,dt,
\end{align*}
since $M$ has the Lebesgue measure $0$. Then the change of variable $z\mapsto \sigma^{\eta}(z)$ reveals that
\begin{align*}
\big(\mathcal{S}_{H}^{j,\eps} f_{\eps}(x)\big)^{2}
\leq
2^{d}
\int_{|z-x|<\sqrt{t}}\big|\delta_{j}T_{t}^{\a,\eps}f_{\eps}(z)\big|^{2}
\frac{w_{\a}(z)}{V_{\sqrt{t}}^{\a}(x)}\,\chi_{\{z\in\R\}}\,dz\,dt.
\end{align*}
Thus, in view of the above estimates, the inequality $V_{\sqrt{t}}^{\a,+}(x)\leq V_{\sqrt{t}}^{\a}(x)$ and the fact that for each $m\in\mathcal{N}_{\eps}$ we have $\langle f_{\eps}, h_{m}^{\a}\rangle_{dw_{\a}}=2^{d}\langle f_{\eps}^{+}, h_{m}^{\a}\rangle_{dw_{\a}^{+}}$ and consequently $\delta_j T_t^{\a,\eps}f_\eps=2^{d}\delta_j T_t^{\a,\eps,+}(f_\eps^{+})$ on $\R$, we get
\begin{align*}
\mathcal{S}_{H}^{j,\eps}f_{\eps}(x)
\leq
2^{{3d}\slash {2}}
S_{H}^{j,\eps,+}(f_{\eps}^{+})(x),\qquad x\in\R.
\end{align*}
Taking into account the symmetry of $\mathcal{S}_{H}^{j,\eps} f_{\eps}$ and $Wdw_{\a}$, we obtain
\begin{align*}
\|S_{H}^{j} (f)\|_{L^{p}(\RR,Wdw_{\a})}
\leq
2^{d\slash p}\sum_{\eps\in\Z}\|\mathcal{S}_{H}^{j,\eps} f_{\eps}\|_{L^{p}(\R,W^{+}dw_{\a}^{+})}
\lesssim
\sum_{\eps\in\Z}\|S_{H}^{j,\eps,+} (f_{\eps}^{+})\|_{L^{p}(\R,W^{+}dw_{\a}^{+})}
\end{align*}
and similarly
$$
\int_{\left\{x\in\RR : S_{H}^{j} (f)(x)>\lambda \right\}}W(y)\, dw_{\a}(y) \le 2^d \sum_{\eps\in\Z}
\int_{\left\{x\in\R : S_{H}^{j,\eps,+} (f_{\eps}^{+}) (x)>2^{-5d\slash 2}\lambda 
\right\}} W^{+}(y)\, dw_{\a}^{+}(y),
\qquad \lambda >0.
$$
Since we have (see \cite[p.\,6]{NS2} for the unweighted case)
\begin{align*}
\|f\|_{L^p(\RR,Wdw_{\alpha})} \simeq \sum_{\varepsilon \in \mathbb{Z}_2^d} 
\|f^{+}_{\varepsilon}\|_{L^p(\mathbb{R}^d_+,W^{+}dw_{\alpha}^+)},
\end{align*}
this shows that the estimates
\begin{align*}
\|S_{H}^{j,\eps,+} (f_{\eps}^{+})\|_{L^{p}(\R,W^{+}dw_{\a}^{+})}
\lesssim
\|f^{+}_{\varepsilon}\|_{L^p(\mathbb{R}^d_+,W^{+}dw_{\alpha}^+)},\qquad \eps\in\Z,
\end{align*}
imply the estimate
\begin{align*}
\|S_{H}^{j} (f)\|_{L^{p}(\RR,Wdw_{\a})}
\lesssim
\|f\|_{L^{p}(\RR,Wdw_{\a})}.
\end{align*}
and an analogous implication is true for the weighted weak type $(1,1)$ inequalities.

Thus we reduced proving Theorem \ref{main} to showing Theorem \ref{main+}.
The proof of the latter result is based on the general Calder\'on-Zygmund theory.
Clearly, the square functions are not linear, but in the well-known way they can be viewed as vector-valued linear operators, see \cite[Section 2]{Sz}.
In fact, we will show that each of the square functions from Theorem \ref{main+}, viewed as a vector-valued operator, is a Calder\'on-Zygmund operator in the sense of the space of homogeneous type $(\R,dw_{\a}^{+},|\cdot|)$. We shall need a slightly more general version of the Calder\'on-Zygmund theory than the one used in \cite{Sz}. More precisely, here we allow weaker smoothness estimates as indicated below. 

Let $\B$ be a Banach space and $K(x,y)$ be a kernel defined on $\R\times\R\backslash\{(x,y) : x=y\}$ and taking values in $\B$. We say that $K(x,y)$ is a standard kernel in the sense of the space of homogeneous type $(\R,dw_{\a}^{+},|\cdot|)$ if it satisfies the growth estimate
\begin{align}\label{gr}
\|K(x,y)\|_{\B}
&\lesssim
\frac{1}{w_{\alpha}^{+}(B(x,|y-x|))}
\end{align}
and the smoothness estimates
\begin{align}\label{sm1}
\|K(x,y)-K(x',y)\|_{\B}
&\lesssim
\bigg(\frac{|x-x'|}{|x-y|}\bigg)^{\delta} \;
\frac{1}{w_{\alpha}^{+}(B(x,|y-x|))},\qquad |x-y|>2|x-x'|,\\\label{sm2}
\|K(x,y)-K(x,y')\|_{\B}
&\lesssim
\bigg(\frac{|y-y'|}{|x-y|}\bigg)^{\delta} \;
\frac{1}{w_{\alpha}^{+}(B(x,|y-x|))},\qquad |x-y|>2|y-y'|,
\end{align}
for some fixed $\delta>0$; here $B(x,r)$ denotes the ball centered at $x$ and with radius $r$, restricted to $\R$. Notice that the bounds \eqref{sm1} and \eqref{sm2} imply analogous estimates with any $0<\delta '<\delta$ replacing $\delta>0$.

A linear operator $T$ assigning to each $f\in L^2(\R,dw_{\a}^+)$ a measurable $\B$-valued function $Tf$ on $\R$ is a (vector-valued) Calder\'on-Zygmund operator in the sense of the space $(\R,dw_{\alpha}^+,|\cdot|)$ if
\begin{itemize}
    \item[(i)] $T$ is bounded from $L^2(\R,dw_{\a}^+)$ to $L^2_{\B}(\R,dw_{\a}^+)$,
    \item[(ii)] there exists a standard $\B$-valued kernel $K(x,y)$ such that
\begin{align*}
Tf(x)=\int_{\R}K(x,y)f(y)\,dw_{\a}^+(y),\qquad \textrm{a.e.}\,\,\, x\notin \supp f,
\end{align*} 
for every $f\in L^2(\R,dw_{\a}^+)$ vanishing outside a compact set contained in $\R$ (we write shortly $T\sim K(x,y)$ for this kind of association).
\end{itemize}
Here integration of $\B$-valued functions is understood in Bochner's sense, and $L_{\B}^{2}$ is the Bochner-Lebesgue space of all $\B$-valued $dw_{\a}^+$-square integrable functions on $\R$.
It is well known that a large part of the classical theory of Calder\'on-Zygmund operators remains valid, with appropriate adjustments, when the underlying space is of homogeneous type and the associated kernels are vector-valued, see the comments in \cite[p.\,649]{NS1} and references given there.

The following result, combined with the general theory of Calder\'on-Zygmund operators and arguments similar to those from the proof of \cite[Corollary 2.5]{Sz}, implies Theorem \ref{main+}, and thus also Theorem \ref{main} by the reduction reasoning described above.

\begin{thm}\label{CZ}
Assume that $\a\in[-1\slash 2,\infty)^{d}$ and $\eps\in\Z$. Then each of the square functions 
\begin{align*}
g_{V}^{\eps,+},\quad g_{H}^{j,\eps,+},\quad g_{H,*}^{j,\eps,+},\quad
S_{V}^{\eps,+},\quad S_{H}^{j,\eps,+},\quad S_{H,*}^{j,\eps,+},\qquad
j=1,\ldots,d,
\end{align*}
viewed as a vector-valued operator related to either $\B=L^{2}(tdt)$ (the case of $g_{V}^{\eps,+}$), or $\B=L^{2}(dt)$ (the cases of $g_{H}^{j,\eps,+}$ and $g_{H,*}^{j,\eps,+}$), or $\B=L^{2}(A,tdtdz)$ (the case of $S_{V}^{\eps,+}$), or $\B=L^{2}(A,dtdz)$ (the cases of $S_{H}^{j,\eps,+}$ and $S_{H,*}^{j,\eps,+}$), is a Calder\'on-Zygmund operator in the sense of the space of homogeneous type $(\R,dw_{\a}^+,|\cdot|)$.
\end{thm}

The proof of Theorem \ref{CZ} splits naturally into proving the following three results. Showing them will complete the whole reasoning justifying Theorem \ref{main}.

\begin{pro}\label{preogr}
Let $\a\in[-1\slash 2,\infty)^{d}$ and $\eps\in\Z$. Then the square functions from Theorem \ref{CZ} are bounded on $L^2(\R,dw_{\a}^{+})$. Consequently, each of them, viewed as a vector-valued operator, is bounded from $L^2(\R,dw^{+}_{\a})$ to $L^2_{\B}(\R,dw^{+}_{\a})$, where $\B$ is as in Theorem \ref{CZ}.
\end{pro}

Formal computations suggest that $S_{V}^{\eps,+}$, $S_{H}^{j,\eps,+}$, $S_{H,*}^{j,\eps,+}$ are associated with the kernels
\begin{align}\label{kv}
K_{z,t}^{\a,\eps,V}(x,y)&=\partial_{t}\big(G_{t}^{\a,\eps}(x+z,y)\big)\sqrt{\v_{\a}(x,z,t)}
\,\chi_{\{x+z\in\R\}},\\\nonumber
K_{z,t}^{\a,\eps,H,j}(x,y)&=\delta_{j,x}\big(G_{t}^{\a,\eps}(x+z,y)\big)\sqrt{\v_{\a}(x,z,t)}
\,\chi_{\{x+z\in\R\}},\qquad j=1,\ldots,d,\\\nonumber
K_{z,t}^{\a,\eps,H,*,j}(x,y)&=\delta_{j,x}^{*}\big(G_{t}^{\a,\eps}(x+z,y)\big)\sqrt{\v_{\a}(x,z,t)}
\,\chi_{\{x+z\in\R\}},\qquad j=1,\ldots,d,
\end{align}
respectively. A part of the next result shows that this is indeed true, at least in the Calder\'on-Zygmund theory sense.

\begin{pro}\label{prestow}
Let $\a\in [-1\slash 2,\infty)^{d}$ and $\eps\in\Z$. Then the square functions from Theorem \ref{CZ}, viewed as vector-valued linear operators related to $\B$ as in Theorem \ref{CZ}, are associated with the following kernels:
\begin{displaymath}
\begin {array}{lll}
g_{V}^{\eps,+}\sim\big\{\partial_{t}G_{t}^{\a,\eps}(x,y)\big\}_{t>0}
,\qquad
&S_{V}^{\eps,+}\sim\big\{K_{z,t}^{\a,\eps,V}(x,y)\big\}_{(z,t)\in A}
,
&\\
g_{H}^{j,\eps,+}\sim\big\{\delta_{j,x}G_{t}^{\a,\eps}(x,y)\big\}_{t>0}
,\qquad
&S_{H}^{j,\eps,+}\sim\big\{K_{z,t}^{\a,\eps,H,j}(x,y)\big\}_{(z,t)\in A}
,
&\qquad j=1,\ldots,d,\\
g_{H,*}^{j,\eps,+}\sim\big\{\delta_{j,x}^{*}G_{t}^{\a,\eps}(x,y)\big\}_{t>0}
,\qquad
&S_{H,*}^{j,\eps,+}\sim\big\{K_{z,t}^{\a,\eps,H,*,j}(x,y)\big\}_{(z,t)\in A}
,
&\qquad j=1,\ldots,d.
\end {array}
\end{displaymath}
\end{pro}

\begin{thm}\label{kes}
Assume that $\alpha\in [-1\slash 2,\infty)^{d}$ and $\eps\in\Z$. Let $K(x,y)$ be any of the vector-valued kernels listed in Proposition \ref{prestow}. Then $K(x,y)$ satisfies the standard estimates \eqref{gr}, \eqref{sm1} and \eqref{sm2} with the relevant space $\B$ and either $\delta=1$ in the cases of $g$-functions, or $\delta=1\slash 2$ in the cases of area integrals.
\end{thm}

The proofs of Propositions \ref{preogr} and \ref{prestow} are given in Section \ref{o+s} (in fact we show somewhat stronger result than Proposition \ref{preogr}). The proof of Theorem \ref{kes} is the most technical and tricky part of the paper and is located in Section \ref{ker}.

We conclude this section with various comments and remarks related to the main result.
First, we note that our results imply analogous results for $g$-functions emerging from the Poisson semigroup related to the Dunkl harmonic oscillator. To be more precise, consider the semigroup $\{P_{t}^{\a}\}_{t>0}$ generated by $\sqrt{\mathcal{L}_{\alpha}}$,
\begin{align*}
P_t^{\a} f =
e^{-t\sqrt{\mathcal{L}_{\alpha}}} f 
=
\sum_{n=0}^{\infty} e^{-t\sqrt{\ll_{n}^{\a}}}
   \sum_{|m|=n}
\langle f,h_{m}^{\a}\rangle_{dw_{\a}}h_{m}^{\a},
\end{align*}
and the auxiliary operators
\begin{align*}
P_t^{\a,\eps,+} f =
\sum_{n=0}^{\infty} e^{-t\sqrt{\ll_{n}^{\a}}}
    \sum_{\substack{|m|=n\\ m\in \mathcal{N}_{\eps}}}
\langle f,h_{m}^{\a}\rangle_{dw^{+}_{\a}}h_{m}^{\a},\qquad \eps\in\Z.
\end{align*}
Clearly, by the subordination principle,
\begin{align}\label{subpr}
P_t^\a f(x)=\int_0^\infty T_{t^{2}\slash (4u)}^{\a} f(x)\,
\frac{e^{-u}\,du}{\sqrt{\pi u}},\qquad
P_t^{\a,\eps,+} f(x)=\int_0^\infty \,T_{t^{2}\slash (4u)}^{\a,\eps,+} f(x)\,
\frac{e^{-u}\,du}{\sqrt{\pi u}}.
\end{align}
We consider the following $g$-functions:
\begin{displaymath}
\begin {array}{llll}
g_{V,P}(f)(x)&=
\big\|\partial_{t}P_{t}^{\a}f(x)\big\|_{L^{2}(tdt)},\quad
g_{V,P}^{\eps,+}(f)(x)&=
\big\|\partial_{t}P_{t}^{\a,\eps,+}f(x)\big\|_{L^{2}(tdt)},&\\
g_{H,P}^{j}(f)(x)&=
\big\|\delta_{j}P_{t}^{\a}f(x)\big\|_{L^{2}(tdt)},\quad
g_{H,P}^{j,\eps,+}(f)(x)&=
\big\|\delta_{j}P_{t}^{\a,\eps,+}f(x)\big\|_{L^{2}(tdt)},\qquad&
j=1,\ldots,d,\\
g_{H,*,P}^{j}(f)(x)&=
\big\|\delta_{j}^{*}P_{t}^{\a}f(x)\big\|_{L^{2}(tdt)},\quad
g_{H,*,P}^{j,\eps,+}(f)(x)&=
\big\|\delta_{j}^{*}P_{t}^{\a,\eps,+}f(x)\big\|_{L^{2}(tdt)},\qquad&
j=1,\ldots,d.
\end {array}
\end{displaymath}
The result below is a consequence of \eqref{subpr} and Theorems \ref{main} and \ref{main+}.
\begin{thm}\label{Poisson}
Assume that $\a\in[-1\slash 2,\infty)^{d}$ and $W$ is a weight on $\RR$ invariant under the reflections $\sigma_1,\ldots,\sigma_d$. Then each of the $g$-functions 
\begin{align*}
g_{V,P},\quad g_{H,P}^{j},\quad g_{H,*,P}^{j},\qquad j=1,\ldots,d,
\end{align*}
is bounded on $L^{p}(\RR,Wdw_{\a})$, $W^{+}\in A_{p}^{\a,+}$, $1<p<\infty$, and from $L^{1}(\RR,Wdw_{\a})$ to weak $L^{1}(\RR,Wdw_{\a})$, $W^{+}\in A_{1}^{\a,+}$. 
Furthermore, the Laguerre-type square functions
\begin{align*}
g_{V,P}^{\eps,+},\quad g_{H,P}^{j,\eps,+},\quad g_{H,*,P}^{j,\eps,+},\qquad
j=1,\ldots,d,\quad\eps\in\Z,
\end{align*}
are bounded on $L^{p}(\R,U dw_{\a}^{+})$, $U\in A_{p}^{\a,+}$, $1<p<\infty$, and from $L^1(\R,U dw_{\a}^{+})$ to weak $L^{1}(\R,U dw_{\a}^{+})$, $U\in A_{1}^{\a,+}$.
\end{thm}
Treatment of Lusin's area integrals associated to the Poisson semigroup is more subtle. In particular, one cannot apply the arguments from \cite[Section 2]{BMR} since in the present situation the function $V_t^\a(x)$ depends not only on $t$, but also on $x$. 

Next, we note that for the particular $\a=(-1\slash 2,\ldots,-1\slash 2)$ the generalized Hermite functions become the classic Hermite functions and $L_{\a}$ is the Euclidean harmonic oscillator. Thus Theorem \ref{main} provides, in particular, results in the Hermite setting for which certain square functions were studied earlier. To be more precise, the vertical $g$-function $g_V$ was considered by Thangavelu \cite[Chapter 4]{T} to prove the Marcinkiewicz multiplier theorem for Hermite function expansions. The Poisson semigroup based $g$-functions $g_{V,P}$, $g_{H,P}^{j}$, $g_{H,*,P}^{j}$, $j=1,\ldots,d$, were studied by Harboure, de Rosa, Segovia and Torrea \cite{HRST}, in connection with Riesz transforms associated to the Hermite setting. All the abovementioned square functions were reinvestigated later by Stempak and Torrea \cite{StTo2}. Lusin's area integrals for Hermite function expansions were studied recently, in the one-dimensional case, by Betancor, Molina and Rodr\'{\i}guez-Mesa \cite{BMR}. The area integral $g^2_{\mathbb{W}}$ there coincides, up to a multiplicative constant, with our area integral $S_V$ with slightly modified aperture of the parabolic cone $A$ (see also Remark \ref{ext} below).

We now focus on the relation between the Laguerre-type square functions studied in this paper and the Laguerre setting from \cite{Sz}. We note that for the particular $\eps_{0}=(0,\ldots,0)$, the operators $T_{t}^{\a,\eps_{0},+}$, $t>0$, coincides, up to the factor $2^{-d}$, with the Laguerre semigroup $T_{t}^{\a}$ considered in \cite{Sz}. 
Moreover, for $\eps=e_j$, $j=1,\ldots,d$, the operators $T_t^{\a,e_j,+}$ are related to the modified Laguerre semigroups $\widetilde T_{t}^{\a,j}$ (see \cite[Section 2]{Sz} for the definition) by
\begin{align}\label{eq5}
\widetilde T_t^{\a,j}=2^{d}e^{-2t}\,T_{t}^{\a,e_{j},+}.
\end{align}
Therefore many results of \cite{Sz} can be seen as special cases of Theorem \ref{CZ}. More precisely, these observations, or rather analogous observations concerning the integral kernels of the semigroups in question, combined with Theorem \ref{CZ} show that the $g$-functions $g_{V,T}$, $g_{H,T}^{i}$, $g_{H,\widetilde T}^{j,i}$, $i,j=1,\ldots,d$, investigated in \cite{Sz} can be viewed as vector-valued Calder\'on-Zygmund operators. The fact that $g_{V,\widetilde T}^{j}$, $j=1,\ldots,d$, from \cite{Sz} may be interpreted as vector-valued Calder\'on-Zygmund operators can be, in principle, also recovered from the results and reasonings of this paper; this, however, is less explicit because of the factor $e^{-2t}$ in \eqref{eq5}, which does not affect the horizontal $g$-functions.

Further, we define Lusin's area type integrals in the Laguerre function setting of convolution type; such operators were not considered in \cite{Sz}. We adopt the notation from \cite{Sz}, but to avoid a confusion, here we denote the Laguerre heat semigroup by $\mathbb{T}_t^\a$. Let
\begin{displaymath}
\begin {array}{lll}
S_{V,\mathbb{T}}(f)(x)&=
\Big(\int_{A(x)}t\big|\partial_{t}\mathbb{T}_{t}^{\a}f(z)\big|^{2}
\chi_{\{z\in\R\}}
\frac{d\mu_{\a}(z)}{V_{\sqrt{t}}^{\a,+}(x)}
\,dt\Big)^{1\slash 2},&\\
S_{H,\mathbb{T}}^{j}(f)(x)&=
\Big(\int_{A(x)}\big|\delta_{j}\mathbb{T}_{t}^{\a}f(z)\big|^{2}
\chi_{\{z\in\R\}}
\frac{d\mu_{\a}(z)}{V_{\sqrt{t}}^{\a,+}(x)}
\,dt\Big)^{1\slash 2},\qquad&
j=1,\ldots,d,\\
S_{H,\widetilde{T}}^{j,i}(f)(x)&=
\Big(\int_{A(x)}\big|\delta_{i}\widetilde{T}_{t}^{\a,j}f(z)\big|^{2}
\chi_{\{z\in\R\}}
\frac{d\mu_{\a}(z)}{V_{\sqrt{t}}^{\a,+}(x)}
\,dt\Big)^{1\slash 2},\qquad&
i,j=1,\ldots,d,\quad i\ne j,\\
S_{H,\widetilde{T}}^{j,j}(f)(x)&=
\Big(\int_{A(x)}\big|\delta_{j}^{*}\widetilde{T}_{t}^{\a,j}f(z)\big|^{2}
\chi_{\{z\in\R\}}
\frac{d\mu_{\a}(z)}{V_{\sqrt{t}}^{\a,+}(x)}
\,dt\Big)^{1\slash 2},\qquad&
j=1,\ldots,d,
\end {array}
\end{displaymath}
where $V_t^{\a,+}(x)$ is defined by \eqref{defV+}, because $d\mu_\a\equiv dw_\a^+$. Thus Theorems \ref{main+} and \ref{CZ} provide, in particular, the following result for the Laguerre area integrals.
\begin{thm}\label{Lag}
Assume that $\a\in[-1\slash 2,\infty)^{d}$. Then each of the Lusin's area type integrals
\begin{align*}
S_{V, \mathbb T}, \quad S_{H,\mathbb{T}}^{j},\quad S_{H,\widetilde{T}}^{j,i},
\qquad
i,j=1,\ldots,d,
\end{align*}
viewed as a vector-valued operator related to either $\B=L^{2}(A,tdtdz)$ (the case of $S_{V,\mathbb{T}}$), or $\B=L^{2}(A,dtdz)$ (the cases of $S_{H,\mathbb{T}}^{j}$ and $S_{H,\widetilde{T}}^{j,i}$), is a Calder\'on-Zygmund operator in the sense of the space of homogeneous type $(\R,d\mu_{\a},|\cdot|)$. Consequently, these square functions are bounded on $L^{p}(\R,U d\mu_{\a})$, $U\in A_{p}^{\a,+}$, $1<p<\infty$, and from $L^1(\R,U d\mu_{\a})$ to weak $L^{1}(\R,U d\mu_{\a})$, $U\in A_{1}^{\a,+}$.
\end{thm}

\begin{rem}\label{lowerest}
Theorem \ref{main}, Theorem \ref{main+}, the first identity of Proposition \ref{ogr1} and the analogous equalities for $g_V$, $g_{V,P}^{\eps,+}$ and $g_{V,P}$, together with standard arguments, see \cite[Remark 2.6]{Sz}, allow to show also lower weighted $L^p$ estimates for the vertical $g$-functions under consideration. With the assumption $\a\in [-1\slash 2,\infty)^d$, for $\eps\in\Z$ and $U\in A_{p}^{\a,+}$, $1<p<\infty$, we have
\begin{displaymath}
\begin {array}{lll}
\|f\|_{L^{p}(\R,Udw_{\a}^{+})}
\lesssim&
\|g_{V}^{\eps,+}(f)\|_{L^{p}(\R,Udw_{\a}^{+})},&\qquad f\in L^{p}(\R,Udw_{\a}^{+}),\\
\|f\|_{L^{p}(\R,Udw_{\a}^{+})}
\lesssim&
\|g_{V,P}^{\eps,+}(f)\|_{L^{p}(\R,Udw_{\a}^{+})},&\qquad f\in L^{p}(\R,Udw_{\a}^{+}).
\end {array}
\end{displaymath}
Consequently, if $W$ is a symmetric weight on $\RR$, $W^{+}\in A_{p}^{\a,+}$, $1<p<\infty$, we also have 
\begin{displaymath}
\begin {array}{lll}
\|f\|_{L^{p}(\RR,Wdw_{\a})}
\lesssim&
\|g_{V}(f)\|_{L^{p}(\RR,Wdw_{\a})},&\qquad f\in L^{p}(\RR,Wdw_{\a}),\\
\|f\|_{L^{p}(\RR,Wdw_{\a})}
\lesssim&
\|g_{V,P}(f)\|_{L^{p}(\RR,Wdw_{\a})},&\qquad f\in L^{p}(\RR,Wdw_{\a}).
\end {array}
\end{displaymath}
\end{rem}

\begin{rem}\label{ext}
The exact aperture of the parabolic cone $A$ is not essential for our developments. Indeed, if we fix $\beta>0$ and write $A_{\beta}=\Big\{(z,t)\in \RR\times(0,\infty) : |z|<\beta\sqrt{t}\Big\}$ instead of $A$ in the definitions of Lusin's area type integrals, then the results of this paper, and in particular Theorem \ref{main}, remain valid. 
\end{rem}
\section{$L^{2}$-Boundedness and Kernel associations}\label{o+s}
\setcounter{equation}{0}

In this section we check that the Laguerre-type square functions under consideration are bounded on the Hilbert space $L^2(\R,dw^{+}_{\a})$. We also show that these square functions, viewed as vector-valued operators, are associated with the relevant kernels.

The following result is essentially a slight generalization of \cite[Proposition 3.1]{Sz} and \cite[Proposition 3.2]{Sz}. The proof is nearly identical and thus is omitted. A crucial fact needed in the proof is that for each $\eps\in\Z$ the system $\{2^{d\slash 2}h_{m}^{\a} : m\in\mathcal{N}_{\eps}\}$ is an orthonormal basis in $L^2(\R,dw^{+}_{\a})$.

\begin{pro}\label{ogr1}
Assume that $\a\in [-1\slash 2,\infty)^{d}$ and $\eps\in\Z$. Then

\begin {align}\nonumber
\|g_{V}^{\eps,+}(f)\|_{L^2(\R,dw^{+}_{\a})}
=&\,
2^{-d-1}\|f\|_{L^2(\R,dw^{+}_{\a})},\qquad f\in L^2(\R,dw^{+}_{\a}),\\\label{iden1}
\Big\|\big|\big(g_{H}^{1,\eps,+}(f),\ldots,g_{H}^{d,\eps,+}(f)\big)\big|_{\ell^2}\Big\|_{L^2(\R,dw^{+}_{\a})}
\lesssim&\,
\|f\|_{L^2(\R,dw^{+}_{\a})},\qquad
f\in L^2(\R,dw^{+}_{\a}),\\\nonumber
\Big\|\big|\big(g_{H,*}^{1,\eps,+}(f),\ldots,g_{H,*}^{d,\eps,+}(f)\big)\big|_{\ell^2}\Big\|_{L^2(\R,dw^{+}_{\a})}
\simeq&\,
\|f\|_{L^2(\R,dw^{+}_{\a})},\qquad
f\in L^2(\R,dw^{+}_{\a}).
\end {align}
Moreover, if $\eps\ne(0,\ldots,0)$, then the relation "$\lesssim$" in \eqref{iden1} can be replaced by "$\simeq$". The same is true for $\eps=(0,\ldots,0)$ provided that $f$ is taken from the subspace $\{h_{(0,\ldots,0)}^{\a}\}^{\perp}\subset L^2(\R,dw_{\a}^+)$.
\end{pro}

\begin{pro}\label{ogr2}
Assume that $\a\in [-1\slash 2,\infty)^{d}$ and $\eps\in\Z$. Then
\begin {align}\nonumber
\|S_{V}^{\eps,+}(f)\|_{L^2(\R,dw^{+}_{\a})}
\simeq&\,
\|f\|_{L^2(\R,dw^{+}_{\a})},\qquad f\in L^2(\R,dw^{+}_{\a}),\\\label{iden2}
\Big\|\big|\big(S_{H}^{1,\eps,+}(f),\ldots,S_{H}^{d,\eps,+}(f)\big)\big|_{\ell^2}\Big\|_{L^2(\R,dw^{+}_{\a})}
\lesssim&\,
\|f\|_{L^2(\R,dw^{+}_{\a})},\qquad
f\in L^2(\R,dw^{+}_{\a}),\\\nonumber
\Big\|\big|\big(S_{H,*}^{1,\eps,+}(f),\ldots,S_{H,*}^{d,\eps,+}(f)\big)\big|_{\ell^2}\Big\|_{L^2(\R,dw^{+}_{\a})}
\simeq&\,
\|f\|_{L^2(\R,dw^{+}_{\a})},\qquad
f\in L^2(\R,dw^{+}_{\a}).
\end {align}
Moreover, if $\eps\ne(0,\ldots,0)$, then the relation "$\lesssim$" in \eqref{iden2} can be replaced by "$\simeq$". The same is true for $\eps=(0,\ldots,0)$ provided that $f$ is taken from the subspace $\{h_{(0,\ldots,0)}^{\a}\}^{\perp}\subset L^2(\R,dw_{\a}^+)$.
\end{pro}

\begin {proof}
We give a justification only for the first relation. The remaining cases, being similar, are left to the reader. 
Using the Fubini-Tonelli theorem, the estimate \eqref{esv} of $V_{\sqrt{t}}^{\a,+}(x)$ and then the inequalities
\begin{align*}
\int_{0}^{\infty}\chi_{\{|x_{j}-z_{j}|<\sqrt{t}\}}\,
\frac{x_{j}^{2\a_{j}+1}}{\sqrt{t}(x_{j}+\sqrt{t})^{2\a_{j}+1}}\,dx_{j}\leq 2,
\end{align*}
which is legitimate since the integrand is dominated by $t^{-1\slash 2}$, and 
\begin{align*}
\int_{0}^{\infty}\chi_{\{|x_{j}-z_{j}|<\sqrt{t}\}}\,
\frac{x_{j}^{2\a_{j}+1}}{\sqrt{t}(x_{j}+\sqrt{t})^{2\a_{j}+1}}\,dx_{j}
\geq
t^{-1\slash 2}
\int_{z_{j}+\sqrt{t}\slash 2}^{z_{j}+\sqrt{t}}
\bigg(\frac{x_j}{x_j+\sqrt{t}}\bigg)^{2\a_{j}+1}\,dx_j
\geq
\frac{1}{2}3^{-2\a_j-1},
\end{align*}
which holds because the function $x_j\mapsto \Big(\frac{x_j}{x_j+\sqrt{t}}\Big)^{2\a_j+1}$ is increasing for $x_j>0$, we obtain
\begin{align*}
\|S_{V}^{\eps,+}(f)\|_{L^2(\R,dw^{+}_{\a})}^{2}
&\simeq
\int_{\R}\int_{0}^{\infty}t|\partial_{t}T_{t}^{\a,\eps,+}f(z)|^{2}\,dt\,dw^{+}_{\a}(z)
=\|g_{V}^{\eps,+}(f)\|_{L^2(\R,dw^{+}_{\a})}^{2}.
\end{align*}
Now the conclusion follows from the first identity of Proposition \ref{ogr1}.
\end {proof}
Proposition \ref{ogr1} together with Proposition \ref{ogr2} imply Proposition \ref{preogr}.

Next we prove that each of the Laguerre-type square functions under consideration, viewed as a vector-valued linear operator, is indeed associated with the relevant kernel in the sense of the Calder\'on-Zygmund theory. We adapt essentially the reasoning given in the proof of \cite[Proposition 2.3]{Sz}, see also comments and references given there.

\begin {proof}[Proof of Proposition \ref{prestow}.]
A careful repetition of the arguments given in the proof of \cite[Proposition 2.3]{Sz}, see also \cite[Section 2]{StTo2}, leads to the desired conclusions for the $g$-functions $g_{V}^{\eps,+}$, $g_{H}^{j,\eps,+}$ and $g_{H,*}^{j,\eps,+}$, since we have a suitable estimate for the generalized Hermite functions, see \eqref{oszh1}, and the relevant kernel estimates, see Theorem \ref{kes}.

Treatment of the area integrals $S_{V}^{\eps,+}$, $S_{H}^{j,\eps,+}$ and $S_{H,*}^{j,\eps,+}$, is slightly different, but relies on similar arguments. Hence we give the details only in the case of $S_{V}^{\eps,+}$, leaving the remaining cases to the reader. Let $\B=L^2(A,tdtdz)$. Proceeding as in the proof of \cite[Proposition 2.3]{Sz} one reduces the task to checking that
\begin{align}\nonumber
\Big\langle&\Big\{\partial_{t}T_{t}^{\a,\eps,+} f(x+z)\chi_{\{x+z\in\R\}}\sqrt{\v_{\a}(x,z,t)}\Big\}_{(z,t)\in A},r\Big\rangle_{L^2_{\B}(\R,dw_{\a}^{+})}\\\label{eqs}
&=
\bigg\langle\int_{\R}\big\{K_{z,t}^{\a,\eps,V}(x,y)\big\}_{(z,t)\in A}f(y)
\,dw^{+}_{\a}(y),r\bigg\rangle_{L^2_{\B}(\R,dw_{\a}^{+})}
\end{align}
for every $f\in C^{\infty}_c(\R)$ and $r(x,z,t)=r_{1}(x)r_{2}(z,t)$, where $r_{1}\in C^{\infty}_c(\R)$,  $r_{2}\in C^{\infty}_c(A)$ and $\supp f\cap \supp r_{1}=\emptyset$. We first deal with the left-hand side of \eqref{eqs},
\begin{align*}
\Big\langle&\Big\{\partial_{t}T_{t}^{\a,\eps,+} f(x+z)\chi_{\{x+z\in\R\}}\sqrt{\v_{\a}(x,z,t)}\Big\}_{(z,t)\in A},r\Big\rangle_{L^2_{\B}(\R,dw_{\a}^{+})}\\
=&\,
\int_{A}t\overline{r_{2}(z,t)}\int_{\R}
\partial_{t}T_{t}^{\a,\eps,+}f(x+z)\chi_{\{x+z\in\R\}}\sqrt{\v_{\a}(x,z,t)}
\,\overline{r_{1}(x)}\,dw^{+}_{\a}(x)\,dt\,dz\\
=&
-
\int_{A}t\overline{r_{2}(z,t)}\int_{\R}\bigg(\sum_{n=0}^{\infty}\ll_{n}^{\a}
\,e^{-t\ll_{n}^{\a}}
\sum_{\substack{|m|=n\\ m\in \mathcal{N}_{\eps}}}
\langle f,h_{m}^{\a}\rangle_{dw^{+}_{\a}}h_{m}^{\a}(x+z)\bigg)\chi_{\{x+z\in\R\}}\\
&\times
\sqrt{\v_{\a}(x,z,t)}\,
\overline{r_{1}(x)}\,dw^{+}_{\a}(x)\,dt\,dz.
\end{align*}
The first identity above follows by Fubini's theorem; the possibility of its application can be justified with the aid of the boundedness of $S_{V}^{\eps,+}$ on $L^2(\R,dw^{+}_{\a})$. The second equality is obtained by exchanging the order of $\partial_{t}$ and $\sum$, which is legitimate in view of \eqref{oszh1}.

Now we focus on the right-hand side of \eqref{eqs}. Changing the order of integrals, which is justified by the growth condition for the kernel $\big\{K_{z,t}^{\a,\eps,V}(x,y)\big\}$, see Theorem \ref{kes}, and using the fact that the supports of $f$ and $r_{1}$ are disjoint and compact, we see that the expression in question is equal
\begin{align*}
\int_{A}t\overline{r_{2}(z,t)}\int_{\R}\int_{\R}K_{z,t}^{\a,\eps,V}(x,y)f(y)
\overline{r_{1}(x)}\,dw^{+}_{\a}(y)\,dw^{+}_{\a}(x)\,dt\,dz.
\end{align*}
Then expressing $K_{z,t}^{\a,\eps,V}$ by means of the series and then using Fubini's theorem, whose application is legitimate in view of \eqref{oszh1}, we get
\begin{align*}
\int_{\R}&K_{z,t}^{\a,\eps,V}(x,y)f(y)
\,dw^{+}_{\a}(y)\\
=&
-
\int_{\R}\bigg(\sum_{n=0}^{\infty}\ll_{n}^{\a}\,e^{-t\ll_{n}^{\a}}
\sum_{\substack{|m|=n\\ m\in \mathcal{N}_{\eps}}}
h_{m}^{\a}(x+z)h_{m}^{\a}(y)\bigg)\chi_{\{x+z\in\R\}}
\sqrt{\v_{\a}(x,z,t)}
f(y)\,dw^{+}_{\a}(y)\\
=&
-
\sum_{n=0}^{\infty}\ll_{n}^{\a}\,e^{-t\ll_{n}^{\a}}
\sum_{\substack{|m|=n\\ m\in \mathcal{N}_{\eps}}}
\langle f,h_{m}^{\a}\rangle_{dw^{+}_{\a}}
h_{m}^{\a}(x+z)\chi_{\{x+z\in\R\}}
\sqrt{\v_{\a}(x,z,t)}.
\end{align*}
Integrating the last identity against $t\overline{r_{1}(x)r_{2}(z,t)}\,dw_{\a}^{+}(x)\,dt\,dz$, we see that both sides of \eqref{eqs} coincide.
\end {proof}

\section{Kernel estimates}\label{ker}
\setcounter{equation}{0}
This section is devoted to the proofs of the relevant kernel estimates for all the considered square functions. We generalize the arguments developed in \cite{NS1,NS3}, which are based on Schl\"afli's integral representation for the modified Bessel function $I_{\nu}$ involved in the Dunkl heat kernel. This method was used also by the author in \cite{Sz} to obtain the standard estimates for the kernel $G_{t}^{\a,\eps}(x,y)$ in the extreme case when $\eps=(0,\ldots,0)$. Recall that we always assume that $\a\in[-1\slash 2,\infty)^{d}$.

Given $\eps\in\Z$, the $\eps$-component of the Dunkl heat kernel is given by, see \cite[Section 5]{NS3},
\begin{equation}\label{hk}
G_{t}^{\a,\eps}(x,y)=\frac{1}{2^{d}}\Big(\frac{1-\zeta^{2}}{2\zeta}\Big)^{d+|\alpha|+|\eps|}(xy)^{\eps}
\int_{[-1,1]^{d}}\exp\Big(-{\frac{1}{4\zeta}}q_{+}(x,y,s)-\frac{\zeta}{4}q_{-}(x,y,s)\Big)\,\Pi_{\a+\eps}(ds),
\end{equation}
where $(xy)^{\eps}=(x_{1}y_{1})^{\eps_{1}}\cdot\ldots\cdot (x_{d}y_{d})^{\eps_{d}}$,
\begin{align*}
q_{\pm}(x,y,s)=
|x|^2+|y|^2\pm 2\sum_{i=1}^{d}x_{i}y_{i}s_{i},
\end{align*}
and $t>0$ and $\z\in(0,1)$ are related by $\zeta=\tanh t$; equivalently
\begin{equation}\label{zz}
t=t(\z)=\frac{1}{2}\log\frac{1+\z}{1-\z}.
\end{equation}
The measure $\Pi_{\beta}$ appearing in \eqref{hk} is a product of one-dimensional measures, $\Pi_{\beta}=\bigotimes_{i=1}^{d}\Pi_{\beta_{i}}$, where $\Pi_{\beta_{i}}$ is given by the density
$$
\Pi_{\beta_{i}}(ds_{i}) = \frac{(1-s_{i}^2)^{\beta_{i}-1\slash 2}ds_{i}}{\sqrt{\pi} 2^{\beta_{i}}\Gamma{(\beta_{i}+1\slash 2)}},
    \qquad \beta_{i} > -1\slash 2,
$$
and in the limiting case of $\beta_{i}=-1\slash 2$, $\Pi_{-1\slash 2}=\big( \eta_{-1} + \eta_1 \big)\slash \sqrt{2\pi}$, with $\eta_{-1}$ and $\eta_{1}$ denoting the point masses at $-1$ and $1$, respectively.

To estimate expressions related to $G_{t}^{\a,\eps}(x,y)$ we will use several technical lemmas which are gathered below. Some of them were obtained elsewhere, but we state them anyway for the sake of completeness and reader's convenience.

To begin with, notice that we have the asymptotics
\begin{equation}\label{cal}
\log\frac{1+\zeta}{1-\zeta}\sim \zeta,\quad\zeta\to 0^{+}
\qquad \textrm{and}\qquad
\log\frac{1+\zeta}{1-\zeta}\sim -\log(1-\zeta),\quad\zeta\to 1^{-}.
\end{equation}

The following result is a compilation of \cite[Lemma 4.1, Lemma 4.2, Lemma 4.4]{Sz}.

\begin{lem}\label{comp}
Let $b\geq 0$ and $c>0$ be fixed. Then for any $j=1,\ldots,d,$ we have
\begin{align*}
& \emph{(a)} \qquad
	|x_{j}\pm y_{j}s_{j}|\leq \sqrt{q_{\pm}(x,y,s)}\qquad \textrm{and}\qquad
|y_{j}\pm x_{j}s_{j}|\leq \sqrt{q_{\pm}(x,y,s)},\\
& \emph{(b)} \qquad
\big(q_{\pm}(x,y,s)\big)^{b}\exp\big(-cAq_{\pm}(x,y,s)\big)\lesssim A^{-b}\exp\Big(\frac{-cA}{2}q_{\pm}(x,y,s)\Big),\\
& \emph{(c)} \qquad
\int_{0}^{1}\zeta^{-3}\log\frac{1+\zeta}{1-\zeta}\,\exp\Big(-\frac{c}{\zeta}q_{+}(x,y,s)\Big)\,d\zeta
\lesssim \frac{1}{q_{+}(x,y,s)},
\end{align*}
uniformly in $x,y \in \R$, $s \in [-1,1]^d$, and also in $A>0$ if (b) is considered.
\end{lem}

\begin{lem}$($\cite[Lemma 1.1]{StTo}$)$ \label{lem5.4}
Given $a>1$, we have
\begin{equation*}
\int_0^1\zeta^{-a}\exp(-T\zeta^{-1})\,d\zeta\lesssim T^{-a+1},
\qquad T>0.
\end{equation*}
\end{lem}

The following result is a slight extension of \cite[Lemma 4.5]{Sz}, the proof being nearly identical.

\begin{lem}\label{lemma4.5}
If $x,y,z\in\R$ are such that $|x-y|>2|x-z|$, then 
\begin{align*}
\frac{1}{4}q_{\pm}(x,y,s)
\leq
q_{\pm}(z,y,s)
\leq
4q_{\pm}(x,y,s),\qquad s\in[-1,1]^d.
\end{align*}
The same holds after exchanging the roles of $x$ and $y$.
\end{lem} 

\begin{lem}$($\cite[Lemma 5.3]{NS3}, \cite[Lemma 4]{NS2}$)$ \label{lem4}
Assume that $\alpha \in [-1\slash 2, \infty)^d$ and let $\delta,\kappa \in [0,\infty)^d$ be fixed.
Then for $x,y\in\R$, $x \neq y$,
\begin{equation*}
(x+y)^{2\delta} \int_{[-1,1]^d}
	 \big(q_{+}(x,y,s)\big)^{-d - |\alpha| -|\delta|}\,\Pi_{\alpha+\delta+\kappa}(ds)
\lesssim \frac{1}{w_{\alpha}^+(B(x,|y-x|))}
\end{equation*}
and
\begin{equation*}
 (x+y)^{2\delta}\int_{[-1,1]^d} 
    \big(q_{+}(x,y,s)\big)^{-d - |\alpha| -|\delta| - 1\slash 2}\,\Pi_{\alpha+\delta+\kappa}(ds)
\lesssim \frac{1}{|x-y|} \;
 \frac{1}{w_{\alpha}^+(B(x,|y-x|))}.
\end{equation*}
\end{lem}

\begin{lem}\label{F2}
Let $\gamma>0$ be fixed. On the set $\{(x,y,z)\in\R\times\R\times\R : |x-y|>2|x-z|\}$ we have  
\begin{align*}
\bigg(\frac{1}{|z-y|}\bigg)^{\gamma} \;
\frac{1}{w_{\alpha}^{+}(B(z,|z-y|))}
\simeq
\bigg(\frac{1}{|x-y|}\bigg)^{\gamma} \;
\frac{1}{w_{\alpha}^{+}(B(x,|y-x|))}.
\end{align*}
\end{lem}

\begin {proof}
Observe that
\begin{align*}
\frac{1}{2}|y-x|
\leq
|y-x|-|x-z|
\leq
|y-z|\leq
|y-x|+|x-z|
\leq\frac{3}{2}|y-x|.
\end{align*}
Now the conclusion is an easy consequence of the doubling property of the measure $w_{\a}^{+}$.
\end {proof}

To state the next lemma, and also to perform the relevant kernel estimates, we will use the same abbreviations as in \cite{Sz},
$$
\lo(\zeta)=\log\frac{1+\zeta}{1-\zeta}, \qquad \e\b=\exp\Big(-{\frac{1}{4\zeta}}q_{+}(x,y,s)-\frac{\zeta}{4}q_{-}(x,y,s)\Big).
$$
Furthermore, we will often neglect the set of integration $[-1,1]^{d}$ in integrals against $\Pi_{\alpha}$ and write shortly $q_{\pm}$ omitting the arguments.

\begin{lem}\label{lemat}
Assume that $\a\in[-1\slash 2,\infty)^{d}$ and $\xi,\rho,\eps\in\Z$ are fixed and such that $\xi\leq\eps$, $\rho\leq\eps$. Given $C>0$ and $u\in \mathbb{R}$, consider the function acting on $\R\times\R\times (0,1)$ and defined by
\begin{align*}
p_{u}(x,y,\z)=
\sqrt{1-\z^2}\,\z^
{-d-|\a|-|\eps|+|\xi|\slash 2+|\rho|\slash 2-u\slash 2}
\,x^{\eps-\xi}y^{\eps-\rho}\int_{[-1,1]^{d}}\big(\e\b\big)^{C}\,\Pi_{\a+\eps}(ds).
\end{align*}
\begin{itemize}
\item[(a)] If $u\geq 1$, then we have the estimate
\begin{align*}
\|p_{u}(x,y,\z(t))\|_{L^2(dt)}
&\lesssim
\frac{1}{|x-y|^{u-1}} \;
\frac{1}{w^{+}_{\alpha}(B(x,|y-x|))},\qquad x\ne y,
\end{align*}
where $t$ and $\z$ are related as in $\eqref{zz}$.
\item[(b)] If $u\geq 2$, then we also have
\begin{align*}
\|p_{u}(x,y,\z(t))\|_{L^2(tdt)}
&\lesssim
\frac{1}{|x-y|^{u-2}} \;
\frac{1}{w^{+}_{\alpha}(B(x,|y-x|))},\qquad x\ne y.
\end{align*}
\end{itemize}
\end{lem}

\begin {proof}
We start with proving the first estimate.  
Changing the variable according to \eqref{zz} and then using sequently the Minkowski integral inequality, Lemma \ref{comp} (b) (applied with $b=2d+2|\alpha|+2|\eps|-|\xi|-|\rho|+u-2$, $c=C\slash 4$, $A=\zeta^{-1}$), Lemma \ref{lem5.4} (with $a=2$ and $T=\frac{Cq_{+}}{4}$) and the inequality $|x-y|^2\leq q_{+}$, we obtain
\begin{align*}
\|p&_{u}(x,y,\z(t))\|_{L^2(dt)}\\
&=
x^{\eps-\xi}y^{\eps-\rho}
\bigg(\int_{0}^{1}\Big(\frac{1}{\z}\Big)^{2d+2|\a|+2|\eps|-|\xi|-|\rho|+u}
\bigg(\int\big(\e\b\big)^{C}\,\Pi_{\a+\eps}(ds)\bigg)^{2}\,d\z
\bigg)^{1\slash 2}\\
&\leq
x^{\eps-\xi}y^{\eps-\rho}
\int\bigg(\int_{0}^{1}\Big(\frac{1}{\z}\Big)^{2d+2|\a|+2|\eps|-|\xi|-|\rho|+u}
\big(\e\b\big)^{2C}\,d\z
\bigg)^{1\slash 2}\,\Pi_{\a+\eps}(ds)\\
&\lesssim
x^{\eps-\xi}y^{\eps-\rho}
\int(q_{+})^{-d-|\a|-|\eps|+|\xi|\slash 2+|\rho|\slash 2+1-u\slash 2}
\bigg(\int_{0}^{1}\Big(\frac{1}{\z}\Big)^{2}\big(\e\b\big)^{C}\,d\z
\bigg)^{1\slash 2}\,\Pi_{\a+\eps}(ds)\\
&\lesssim
x^{\eps-\xi}y^{\eps-\rho}
\int(q_{+})^{-d-|\a|-|\eps|+|\xi|\slash 2+|\rho|\slash 2+1\slash 2-u\slash 2}
\,\Pi_{\a+\eps}(ds)\\
&\lesssim
\frac{1}{|x-y|^{u-1}}
(x+y)^{2\eps-\xi-\rho}
\int(q_{+})^{-d-|\a|-|\eps-\xi\slash 2-\rho\slash 2|}
\,\Pi_{\a+\eps}(ds).
\end{align*}
Now an application of Lemma \ref{lem4} (with $\delta=\eps-\xi\slash 2-\rho\slash 2$ and $\kappa=\xi\slash 2+\rho\slash 2$) leads to the desired conclusion.

Similar arguments (using this time Lemma \ref{comp} (c) instead of Lemma \ref{lem5.4}) justify the second estimate.
\end {proof}

\begin{lem}$($\cite[Lemma 4.7]{Sz}$)$\label{wch}
Let $F\colon (0,\infty)\times[-1,1]^{d}\to\mathbb{R}$ be a function such that $F(\cdot,s)$ is continuously differentiable for each fixed $s$, and $F(z,\cdot)\in L^1(\Pi_{\a}(ds))$ for any $z>0$. Further, assume that for each $v>0$ there exists $a<v<b$ and a function $f_{a,b}\in L^{1}(\Pi_{\a}(ds))$ such that $|\partial_{z}F(z,s)|\leq f_{a,b}(s)$ for all $z\in[a,b]$ and $s\in[-1,1]^{d}$. Then 
$$
\partial_{z}\int_{[-1,1]^{d}}F(z,s)\,\Pi_{\a}(ds)=
\int_{[-1,1]^{d}}\partial_{z}F(z,s)\,\Pi_{\a}(ds),\qquad z>0.
$$
\end{lem}

In what follows it is convenient to use the following notation.
Given $x,y\in\R$, we write $x\leq y$ if $x_{j}\leq y_{j}$ for each $j=1,\ldots,d$. We denote by $\max\{x,y\}$ the point in $\R$ having the coordinates $\max\{x_{j},y_{j}\}$, $j=1,\ldots,d$, and similarly for $\min\{x,y\}$.

\subsection{Vertical $g$-function based on $\mathbf{\{T_{t}^{\alpha,\eps,+}\}}$}\label{k1}

\begin {proof}[Proof of Theorem \ref{kes}; the case of $g_{V}^{\eps,+}$.]
We first deal with the growth condition. Differentiating \eqref{hk} in $t$ (passing with $\partial_{t}$ under the integral sign can be easily justified by Lemma \ref{wch}, see \cite[Section 4]{Sz}) we get
\begin{align}\label{eq1}
\partial_{t}G_{t}^{\alpha,\eps}(x,y)=
-\frac{1}{2^{d}}(xy)^{\eps}\Big(\frac{1-\z^2}{2\z}\Big)^{d+|\a|+|\eps|}h(x,y,\z),
\end{align}
where the auxiliary function $h$ is given by
\begin{align*}
h(x,y,\z)=&(d+|\a|+|\eps|)\frac{1+\z^2}{\z}\int\e\b\,\Pi_{\a+\eps}(ds)\\
&+\frac{1-\z^2}{\z}\int\e\b
\Big[\frac{\z}{4}q_{-}-\frac{1}{4\z}q_{+}\Big]\,\Pi_{\a+\eps}(ds).
\end{align*}
Notice that the function $h$ depends on $\a$ and $\eps$, but to shorten the notation we do not indicate that explicitly (a similar convention will concern other auxiliary functions appearing in the sequel).

Using Lemma \ref{comp} (b) (first with $b=1$, $c=1\slash 4$, $A=\z$ and then with $b=1$, $c=1\slash 4$, $A=\z^{-1}$) we obtain
\begin{align}\nonumber
|h(x,y,\z)|
&\lesssim
\z^{-1}\int\e\b\,\Pi_{\a+\eps}
+\z^{-1}\int\e\b \Big[\z q_{-}+\frac{q_{+}}{\z}\Big]\,\Pi_{\a+\eps}(ds)\\\label{nier1}
&\lesssim
\z^{-1}\int\big(\e\b\big)^{1\slash 2}\,\Pi_{\a+\eps}(ds).
\end{align}
This, in view of \eqref{eq1}, gives
\begin{align}\label{nier2}
|\partial_{t}G_{t}^{\alpha,\eps}(x,y)|\lesssim
\sqrt{1-\z^2}\,\z^{-d-|\a|-|\eps|-1}\,(xy)^{\eps}\int\big(\e\b\big)^{1\slash 2}\,\Pi_{\a+\eps}(ds).
\end{align}
Finally, Lemma \ref{lemat} (b) (specified to $u=2$, $\xi=\rho=0$) leads directly to the desired bound.

We pass to proving the smoothness estimates. By symmetry reasons, it suffices to show that
$$
\big\|\partial_{t}G_{t}^{\a,\eps}(x,y)-\partial_{t}G_{t}^{\a,\eps}(x',y)\big\|_{L^2(tdt)}
\lesssim
\frac{|x-x'|}{|x-y|} \;
\frac{1}{w^{+}_{\a}(B(x,|y-x|))},\qquad |x-y|>2|x-x'|.
$$
By the mean value theorem 
\begin{align*}
\big|\partial_{t}G_{t}^{\a,\eps}(x,y)-\partial_{t}G_{t}^{\a,\eps}(x',y)\big|
\leq&
\,|x-x'|
\big|\nabla_{\!x}\partial_{t}G_{t}^{\a,\eps}(\t,y)\big|,
\end{align*}
where $\t$ is a convex combination of $x$ and $x'$ that depends also on $t$. Thus our task reduces to proving that 
\begin{align*}
\|\partial_{x_{i}}\partial_{t}G_{t}^{\a,\eps}(\t,y)\|_{L^{2}(tdt)}\lesssim
\frac{1}{|x-y|} \;
\frac{1}{w^{+}_{\a}(B(x,|y-x|))},\qquad |x-y|>2|x-x'|,
\end{align*}
for each $i=1,\ldots,d$. To proceed we first analyze the derivative
\begin{align*}
\partial_{x_{i}}\partial_{t}G_{t}^{\a,\eps}(x,y)=
&-\frac{1}{2^{d}}(xy)^{\eps}
\Big(\frac{1-\z^{2}}{2\z}\Big)^{d+|\a|+|\eps|}
\partial_{x_{i}}h(x,y,\z)\\
&-\chi_{\{\eps_{i}=1\}}
\frac{1}{2^{d}}x^{\eps-e_{i}}y^{\eps}
\Big(\frac{1-\z^{2}}{2\z}\Big)^{d+|\a|+|\eps|}
h(x,y,\z).
\end{align*}
An elementary computation shows that
\begin{align*}
\partial_{x_{i}}h(x,y,\z)=&
-(d+|\a|+|\eps|)\frac{1+\z^2}{\z}\int\e\b\Big[
\frac{1}{2\z}(x_{i}+y_{i}s_{i})+\frac{\z}{2}(x_{i}-y_{i}s_{i})\Big]\,\Pi_{\a+\eps}(ds)\\
&-
\frac{1-\z^2}{\z}\int\e\b\Big[
\frac{1}{2\z}(x_{i}+y_{i}s_{i})+\frac{\z}{2}(x_{i}-y_{i}s_{i})\Big]
\Big[\frac{\z}{4}q_{-}-\frac{1}{4\z}q_{+}\Big]\,\Pi_{\a+\eps}(ds)\\
&+
\frac{1-\z^2}{\z}\int\e\b \Big[\frac{\z}{2}(x_{i}-y_{i}s_{i})-\frac{1}{2\z}(x_{i}+y_{i}s_{i})\Big]\,\Pi_{\a+\eps}(ds).
\end{align*}
Applying Lemma \ref{comp} (a) and then repeatedly Lemma \ref{comp} (b) (specified to $b=1\slash 2$ or $b=1$) we get
\begin{align*}
|\partial_{x_{i}}h(x,y,\z)|
\lesssim\,&
\z^{-1}\int\e\b \Big[
\frac{\sqrt{q_{+}}}{\z}+\z \sqrt{q_{-}}\Big]\,\Pi_{\a+\eps}(ds)\\
&+
\z^{-1}\int\e\b \Big[
\frac{\sqrt{q_{+}}}{\z}+\z \sqrt{q_{-}}\Big]
\Big[\z q_{-}+\frac{q_{+}}{\z}\Big]
\,\Pi_{\a+\eps}(ds)\\
\lesssim\,&
\z^{-3\slash 2}\int\big(\e\b\big)^{1\slash 4}\,\Pi_{\a+\eps}(ds).
\end{align*}
Denote $x^{*}=\max\{x,x'\}$ and observe that $\t\leq x^{*}$ and also $|x-x^{*}|\leq|x-x'|$. Then using the last estimate of $\partial_{x_{i}}h(x,y,\z)$,  \eqref{nier1} and then Lemma \ref{lemma4.5} (first with $z=\t$ and then with $z=x^{*}$) produces
\begin{align}\label{nier3}
|\partial&_{x_{i}}\partial_{t}G_{t}^{\a,\eps}(\t,y)|\\\nonumber
\lesssim&
\sqrt{1-\z^2}\,\z^{-d-|\a|-|\eps|-3\slash 2}\,(\t y)^{\eps}\int\big(\e\c\big)^{1\slash 4}\,\Pi_{\a+\eps}(ds)\\\nonumber
&+
\chi_{\{\eps_{i}=1\}}\sqrt{1-\z^{2}}\,\z^{-d-|\a|-|\eps|-1}\,
\t^{\eps-e_{i}}y^{\eps}\int\big(\e\c\big)^{1\slash 2}\,\Pi_{\a+\eps}(ds)\\\nonumber
\lesssim&
\sqrt{1-\z^2}\,\z^{-d-|\a|-|\eps|-3\slash 2}\,(x^{*} y)^{\eps}\int\big(\e\eee\big)^{1\slash 64}\,\Pi_{\a+\eps}(ds)\\\nonumber
&+
\chi_{\{\eps_{i}=1\}}\sqrt{1-\z^{2}}\,\z^{-d-|\a|-|\eps|-1}\,
(x^{*})^{\eps-e_{i}}y^{\eps}\int\big(\e\eee\big)^{1\slash 32}\,\Pi_{\a+\eps}(ds),
\end{align}
provided that $|x-y|>2|x-x'|$. Now Lemma \ref{lemat} (b) (taken with $u=3$, $\xi=\rho=0$ and $\xi=e_{i}$, $\rho=0$) combined with Lemma \ref{F2} (specified to $z=x^{*}$) gives the desired smoothness condition. 

The proof of the case of $g_{V}^{\eps,+}$ in Theorem \ref{kes} is finished.
\end {proof}

\subsection{Horizontal $g$-functions based on  $\mathbf{\{T_{t}^{\a,\eps,+}\}}$}\label{k2}

\begin {proof}[Proof of Theorem \ref{kes}; the case of $g_{H}^{j,\eps,+}$.]
To compute $\delta_{j,x}G_{t}^{\a,\eps}(x,y)$, observe that $\delta_{j,x}$ may be replaced either by $\delta_{j,x}^{e}$ or $\delta_{j,x}^{o}$ (see \cite[p.\,548]{NS3}),
\begin{align*}
\delta_{j,x}^{e}=\partial_{x_{j}}+x_{j},\qquad \delta_{j,x}^{o}=\partial_{x_{j}}+x_{j}+\frac{2\a_{j}+1}{x_{j}},
\end{align*}
depending on whether $\eps_{j}=0$ or $\eps_{j}=1$, respectively. Then we see that
\begin{align*}
\delta_{j,x}G_{t}^{\a,\eps}(x,y)=&\frac{1}{2^d}\Big(\frac{1-\z^2}{2\z}\Big)^{d+|\a|+|\eps|}
h_{j}(x,y,\z),
\end{align*}
where the auxiliary functions $h_{j}$ are given by
\begin{align*}
h_{j}(x,y,\z)=&
-(xy)^{\eps}
\int\e\b\Big[\frac{1}{2\z}(x_{j}+y_{j}s_{j})+\frac{\z}{2}(x_{j}-y_{j}s_{j})\Big]
\,\Pi_{\a+\eps}(ds)\\
&+
x_{j}(xy)^{\eps}\int\e\b\,\Pi_{\a+\eps}(ds)\\
&+
\chi_{\{\eps_{j}=1\}}(2\a_{j}+2)
x^{\eps-e_{j}}y^{\eps}
\int\e\b\,\Pi_{\a+\eps}(ds).
\end{align*}
Using Lemma \ref{comp} (a), the fact that $x_{j}\leq \sqrt{q_{+}}+\sqrt{q_{-}}$ and then Lemma \ref{comp} (b) (taken with $b=1\slash 2$, $A=\z^{-1}$ and $A=\z$) we obtain
\begin{align}\nonumber
|\delta_{j,x}G_{t}^{\a,\eps}(x,y)|
\lesssim&
\sqrt{1-\z^{2}}\,\z^{-d-|\a|-|\eps|}(xy)^{\eps}\int\e\b
\Big[\frac{\sqrt{q_{+}}}{\z}+\z \sqrt{q_{-}}\Big]\,\Pi_{\a+\eps}(ds)\\\nonumber
&+
\sqrt{1-\z^2}\,\z^{-d-|\a|-|\eps|}(xy)^{\eps}\int
\e\b(\sqrt{q_{+}}+\sqrt{q_{-}})\,\Pi_{\a+\eps}(ds)\\\nonumber
&+
\chi_{\{\eps_{j}=1\}}\sqrt{1-\z^2}\,\z^{-d-|\a|-|\eps|}x^{\eps-e_{j}}y^{\eps}
\int\e\b\,\Pi_{\a+\eps}(ds)\\\label{nier8}
\lesssim&
\sqrt{1-\z^{2}}\,\z^{-d-|\a|-|\eps|-1\slash 2}(xy)^{\eps}\int\big(\e\b\big)^{1\slash 2}\,\Pi_{\a+\eps}(ds)\\\nonumber
&+
\chi_{\{\eps_{j}=1\}}\sqrt{1-\z^2}\,\z^{-d-|\a|-|\eps|}x^{\eps-e_{j}}y^{\eps}
\int\e\b\,\Pi_{\a+\eps}(ds).
\end{align}
Now an application of Lemma \ref{lemat} (a) (specified to $u=1$, $\xi=\rho=0$ and $\xi=e_{j}$, $\rho=0$) leads to the growth condition for $\big\{\delta_{j,x}G_{t}^{\a,\eps}(x,y)\big\}_{t>0}$.

To prove the smoothness estimates we first show that
\begin{align*}
\big\|\delta_{j,x}G_{t}^{\a,\eps}(x,y)-\delta_{j,x}G_{t}^{\a,\eps}(x',y)\big\|_{L^{2}(dt)}
&\lesssim
\frac{|x-x'|}{|x-y|} \;
\frac{1}{w^{+}_{\a}(B(x,|y-x|))},\qquad |x-y|>2|x-x'|.
\end{align*}
Using the mean value theorem we get
$$
\big|\delta_{j,x}G_{t}^{\a,\eps}(x,y)-\delta_{j,x}G_{t}^{\a,\eps}(x',y)\big|
\leq
|x-x'|\big|\nabla_{\!x}\delta_{j,x}G_{t}^{\a,\eps}(\t,y)\big|,
$$
where $\t$ is a convex combination of $x$ and $x'$ (notice that $\t$ depends also on $t$). Thus it suffices to show that for any $i,j=1,\ldots,d,$
\begin{align}\label{nier4}
\|\partial_{x_{i}}\delta_{j,x}G_{t}^{\a,\eps}(\t,y)\|_{L^{2}(dt)}
\lesssim
\frac{1}{|x-y|} \;
\frac{1}{w^{+}_{\a}(B(x,|y-x|))},\qquad |x-y|>2|x-x'|.
\end{align}
We shall first estimate $\partial_{x_{i}}h_{j}(x,y,\z)$. It is convenient to distinguish two cases.
\newline
{\bf Case 1:} $\mathbf {i\ne j.}$ An elementary computation produces
\begin{align*}
\partial&_{x_{i}}h_{j}(x,y,\z)\\
=&
\,(xy)^{\eps}\int\e\b\Big[\frac{1}{2\z}(x_{i}+y_{i}s_{i})+\frac{\z}{2}(x_{i}-y_{i}s_{i})\Big]
\Big[\frac{1}{2\z}(x_{j}+y_{j}s_{j})+\frac{\z}{2}(x_{j}-y_{j}s_{j})\Big]
\,\Pi_{\a+\eps}(ds)\\
&-
x_{j}(xy)^{\eps}\int\e\b\Big[\frac{1}{2\z}(x_{i}+y_{i}s_{i})+\frac{\z}{2}(x_{i}-y_{i}s_{i})\Big]
\,\Pi_{\a+\eps}(ds)\\
&-
\chi_{\{\eps_{j}=1\}}\,(2\a_{j}+2)x^{\eps-e_{j}}y^{\eps}\int\e\b
\Big[\frac{1}{2\z}(x_{i}+y_{i}s_{i})+\frac{\z}{2}(x_{i}-y_{i}s_{i})\Big]
\,\Pi_{\a+\eps}(ds)\\
&-
\chi_{\{\eps_{i}=1\}}\,x^{\eps-e_{i}}y^{\eps}\int\e\b
\Big[\frac{1}{2\z}(x_{j}+y_{j}s_{j})+\frac{\z}{2}(x_{j}-y_{j}s_{j})\Big]
\,\Pi_{\a+\eps}(ds)\\
&+
\chi_{\{\eps_{i}=1\}}\,x_{j}x^{\eps-e_{i}}y^{\eps}\int\e\b\,\Pi_{\a+\eps}(ds)\\
&+
\chi_{\{\eps_{i}=1\}}\chi_{\{\eps_{j}=1\}}\,(2\a_{j}+2)x^{\eps-e_{i}-e_{j}}y^{\eps}
\int\e\b\,\Pi_{\a+\eps}(ds).
\end{align*}
Using sequently Lemma \ref{comp} (a), the fact that $x_{j}\leq\sqrt{q_{+}}+\sqrt{q_{-}}$ and then Lemma \ref{comp} (b) (taken with $b=1\slash 2$, $A=\z^{-1}$ and $A=\z$, respectively) we get
\begin{align*}
|\partial_{x_{i}}h_{j}(x,y,\z)|\lesssim&
\,(xy)^{\eps}\int\e\b\Big[\frac{1}{\z}\sqrt{q_{+}}+\z\sqrt{q_{-}}\Big]^{2}\,\Pi_{\a+\eps}(ds)\\
&+
(xy)^{\eps}\int\e\b\big(\sqrt{q_{+}}+\sqrt{q_{-}}\big)\Big[\frac{1}{\z}\sqrt{q_{+}}+\z\sqrt{q_{-}}\Big]
\,\Pi_{\a+\eps}(ds)\\
&+
\chi_{\{\eps_{j}=1\}}\,x^{\eps-e_{j}}y^{\eps}\int\e\b\Big[\frac{1}{\z}\sqrt{q_{+}}+\z\sqrt{q_{-}}\Big]
\,\Pi_{\a+\eps}(ds)\\
&+
\chi_{\{\eps_{i}=1\}}\,x^{\eps-e_{i}}y^{\eps}\int\e\b
\Big[\frac{1}{\z}\sqrt{q_{+}}+\z\sqrt{q_{-}}\Big]
\,\Pi_{\a+\eps}(ds)\\
&+
\chi_{\{\eps_{i}=1\}}\,x^{\eps-e_{i}}y^{\eps}\int\e\b\big(\sqrt{q_{+}}+\sqrt{q_{-}}\big)\,\Pi_{\a+\eps}(ds)\\
&+
\chi_{\{\eps_{i}=1\}}\chi_{\{\eps_{j}=1\}}\,x^{\eps-e_{i}-e_{j}}y^{\eps}
\int\e\b\,\Pi_{\a+\eps}(ds)\\
\lesssim&
\,\z^{-1}\,(xy)^{\eps}\int\big(\e\b\big)^{1\slash 4}\,\Pi_{\a+\eps}(ds)\\
&+
\chi_{\{\eps_{j}=1\}}\,\z^{-1\slash 2}\,x^{\eps-e_{j}}y^{\eps}\int\big(\e\b\big)^{1\slash 2}
\,\Pi_{\a+\eps}(ds)\\
&+
\chi_{\{\eps_{i}=1\}}\,\z^{-1\slash 2}\,x^{\eps-e_{i}}y^{\eps}\int\big(\e\b\big)^{1\slash 2}
\,\Pi_{\a+\eps}(ds)\\
&+
\chi_{\{\eps_{i}=1\}}\chi_{\{\eps_{j}=1\}}\,x^{\eps-e_{i}-e_{j}}y^{\eps}
\int\e\b\,\Pi_{\a+\eps}(ds).
\end{align*}
{\bf Case 2:} $\mathbf {i= j.}$ We have
\begin{align*}
\partial&_{x_{j}}h_{j}(x,y,\z)\\
=&
\,(xy)^{\eps}\int\e\b\Big[\frac{1}{2\z}(x_{j}+y_{j}s_{j})+\frac{\z}{2}(x_{j}-y_{j}s_{j})\Big]^{2}
\,\Pi_{\a+\eps}(ds)\\
&-
\,(xy)^{\eps}\int\e\b\Big[\frac{1}{2\z}+\frac{\z}{2}\Big]
\,\Pi_{\a+\eps}(ds)
+
\,\big(1+\chi_{\{\eps_{j}=1\}}\big)(xy)^{\eps}\int\e\b
\,\Pi_{\a+\eps}(ds)\\
&-
\,x_{j}(xy)^{\eps}\int\e\b\Big[\frac{1}{2\z}(x_{j}+y_{j}s_{j})+\frac{\z}{2}(x_{j}-y_{j}s_{j})\Big]
\,\Pi_{\a+\eps}(ds)\\
&-
\chi_{\{\eps_{j}=1\}}\,(2\a_{j}+3)x^{\eps-e_{j}}y^{\eps}
\int\e\b\Big[\frac{1}{2\z}(x_{j}+y_{j}s_{j})+\frac{\z}{2}(x_{j}-y_{j}s_{j})\Big]
\,\Pi_{\a+\eps}(ds).
\end{align*}
Proceeding similarly as in Case 1 (and using the inequality $\z^{-1}\geq 1$) we obtain
\begin{align*}
|\partial_{x_{j}}h_{j}(x,y,\z)|\lesssim&
\,\z^{-1}\,(xy)^{\eps}\int\big(\e\b\big)^{1\slash 4}
\,\Pi_{\a+\eps}(ds)\\
&+
\chi_{\{\eps_{j}=1\}}\,\z^{-1\slash 2}\,x^{\eps-e_{j}}y^{\eps}
\int\big(\e\b\big)^{1\slash 2}\,\Pi_{\a+\eps}(ds).
\end{align*}
Now using the above estimates of $\partial_{x_{i}}h_{j}(x,y,\z)$, the fact that $\t\leq x^{*}$ and Lemma \ref{lemma4.5} twice (with $z=\t$ and $z=x^{*}$) we see that
\begin{align}\label{nier9}
|\partial&_{x_{i}}\delta_{j,x}G_{t}^{\a,\eps}(\t,y)|\\\nonumber
\lesssim&
\sqrt{1-\z^{2}}\,\z^{-d-|\a|-|\eps|-1}\,(x^{*}y)^{\eps}
\int\big(\e\eee\big)^{1\slash 64}\,\Pi_{\a+\eps}(ds)\\\nonumber
&+
\,\chi_{\{\eps_{j}=1\}}\,\sqrt{1-\z^{2}}\,\z^{-d-|\a|-|\eps|-1\slash 2}
(x^{*})^{\eps-e_{j}}y^{\eps}\int\big(\e\eee\big)^{1\slash 32}\,\Pi_{\a+\eps}(ds)\\\nonumber
&+
\,\chi_{\{\eps_{i}=1\}}\,\sqrt{1-\z^{2}}\,\z^{-d-|\a|-|\eps|-1\slash 2}
(x^{*})^{\eps-e_{i}}y^{\eps}\int\big(\e\eee\big)^{1\slash 32}\,\Pi_{\a+\eps}(ds)\\\nonumber
&+
\chi_{\{i\ne j\}}\chi_{\{\eps_{i}=1\}}\chi_{\{\eps_{j}=1\}}
\,\sqrt{1-\z^{2}}\,\z^{-d-|\a|-|\eps|}(x^{*})^{\eps-e_{i}-e_{j}}y^{\eps}
\int\big(\e\eee\big)^{1\slash 16}\,\Pi_{\a+\eps}(ds),
\end{align}
provided that $|x-y|>2|x-x'|$. From here \eqref{nier4} follows with the aid of Lemma \ref{lemat} (a) (specified to either $u=2$, $\rho=0$ and $\xi=0$ or $\xi=e_{j}$, $\xi=e_{i}$, or $\xi=e_{i}+e_{j}$) and Lemma \ref{F2} (taken with $z=x^{*}$).

The proof will be finished once we show that
\begin{align*}
\big\|\delta_{j,x}G_{t}^{\a,\eps}(x,y)-\delta_{j,x}G_{t}^{\a,\eps}(x,y')\big\|_{L^{2}(dt)}
&\lesssim
\frac{|y-y'|}{|x-y|} \;
\frac{1}{w^{+}_{\a}(B(x,|y-x|))},\qquad |x-y|>2|y-y'|.
\end{align*}
By the mean value theorem it is enough to verify that for any $i,j=1,\ldots,d,$ we have
\begin{align*}
\|\partial_{y_{i}}\delta_{j,x}G_{t}^{\a,\eps}(x,\t)\|_{L^{2}(dt)}
\lesssim
\frac{1}{|x-y|} \;
\frac{1}{w^{+}_{\a}(B(x,|y-x|))},\qquad |x-y|>2|y-y'|,
\end{align*} 
where $\t$ is a convex combination of $y$ and $y'$. When considering $\partial_{y_{i}}h_{j}(x,y,\z)$ again it is natural to distinguish two cases.
\newline
{\bf Case 1:} $\mathbf {i\ne j.}$
A simple computation gives
\begin{align*}
\partial&_{y_{i}}h_{j}(x,y,\z)\\
=&
\,(xy)^{\eps}\int\e\b\Big[\frac{1}{2\z}(x_{j}+y_{j}s_{j})+\frac{\z}{2}(x_{j}-y_{j}s_{j})\Big]
\Big[\frac{1}{2\z}(y_{i}+x_{i}s_{i})+\frac{\z}{2}(y_{i}-x_{i}s_{i})\Big]
\,\Pi_{\a+\eps}(ds)\\
&-
x_{j}(xy)^{\eps}\int\e\b\Big[\frac{1}{2\z}(y_{i}+x_{i}s_{i})+\frac{\z}{2}(y_{i}-x_{i}s_{i})\Big]
\,\Pi_{\a+\eps}(ds)\\
&-
\chi_{\{\eps_{j}=1\}}\,(2\a_{j}+2)x^{\eps-e_{j}}y^{\eps}
\int\e\b\Big[\frac{1}{2\z}(y_{i}+x_{i}s_{i})+\frac{\z}{2}(y_{i}-x_{i}s_{i})\Big]
\,\Pi_{\a+\eps}(ds)\\
&-
\chi_{\{\eps_{i}=1\}}\,x^{\eps}y^{\eps-e_{i}}\int\e\b
\Big[\frac{1}{2\z}(x_{j}+y_{j}s_{j})+\frac{\z}{2}(x_{j}-y_{j}s_{j})\Big]
\,\Pi_{\a+\eps}(ds)\\
&+
\chi_{\{\eps_{i}=1\}}\,x_{j}x^{\eps}y^{\eps-e_{i}}\int\e\b
\,\Pi_{\a+\eps}(ds)\\
&+
\chi_{\{\eps_{i}=1\}}\chi_{\{\eps_{j}=1\}}\,(2\a_{j}+2)x^{\eps-e_{j}}y^{\eps-e_{i}}
\int\e\b
\,\Pi_{\a+\eps}(ds).
\end{align*}
Proceeding as before (see the estimate of $\partial_{x_{i}}h_{j}(x,y,\z)$ above) we obtain
\begin{align*}
|\partial_{y_{i}}h_{j}(x,y,\z)|\lesssim&
\,\z^{-1}\,(xy)^{\eps}\int\big(\e\b\big)^{1\slash 4}
\,\Pi_{\a+\eps}(ds)\\
&+
\chi_{\{\eps_{j}=1\}}\,\z^{-1\slash 2}\,x^{\eps-e_{j}}y^{\eps}\int\big(\e\b\big)^{1\slash 2}
\,\Pi_{\a+\eps}(ds)\\
&+
\chi_{\{\eps_{i}=1\}}\,\z^{-1\slash 2}\,x^{\eps}y^{\eps-e_{i}}\int\big(\e\b\big)^{1\slash 2}
\,\Pi_{\a+\eps}(ds)\\
&+
\chi_{\{\eps_{i}=1\}}\chi_{\{\eps_{j}=1\}}\,x^{\eps-e_{j}}y^{\eps-e_{i}}
\int\e\b\,\Pi_{\a+\eps}(ds).
\end{align*}
{\bf Case 2:} $\mathbf {i= j.}$
It is not hard to check that
\begin{align*}
\partial&_{y_{j}}h_{j}(x,y,\z)\\
=&
\,(xy)^{\eps}\int\e\b\Big[\frac{1}{2\z}(x_{j}+y_{j}s_{j})+\frac{\z}{2}(x_{j}-y_{j}s_{j})\Big]
\Big[\frac{1}{2\z}(y_{j}+x_{j}s_{j})+\frac{\z}{2}(y_{j}-x_{j}s_{j})\Big]
\,\Pi_{\a+\eps}(ds)\\
&+
(xy)^{\eps}\int\e\b\Big[-\frac{1}{2\z}s_{j}+\frac{\z}{2}s_{j}\Big]
\,\Pi_{\a+\eps}(ds)\\
&-
x_{j}(xy)^{\eps}\int\e\b\Big[\frac{1}{2\z}(y_{j}+x_{j}s_{j})+\frac{\z}{2}(y_{j}-x_{j}s_{j})\Big]
\,\Pi_{\a+\eps}(ds)\\
&-
\chi_{\{\eps_{j}=1\}}\,x^{\eps}y^{\eps-e_{j}}\int\e\b
\Big[\frac{1}{2\z}(x_{j}+y_{j}s_{j})+\frac{\z}{2}(x_{j}-y_{j}s_{j})\Big]
\,\Pi_{\a+\eps}(ds)\\
&+
\chi_{\{\eps_{j}=1\}}\,x_{j}x^{\eps}y^{\eps-e_{j}}\int\e\b
\,\Pi_{\a+\eps}(ds)\\
&+
\chi_{\{\eps_{j}=1\}}\,(2\a_{j}+2)(xy)^{\eps-e_{j}}\int\e\b
\,\Pi_{\a+\eps}(ds)\\
&-
\chi_{\{\eps_{j}=1\}}\,(2\a_{j}+2)x^{\eps-e_{j}}y^{\eps}\int\e\b
\Big[\frac{1}{2\z}(y_{j}+x_{j}s_{j})+\frac{\z}{2}(y_{j}-x_{j}s_{j})\Big]
\,\Pi_{\a+\eps}(ds)
\end{align*}
and therefore (see Case 2 in the estimate of $\partial_{x_{i}}h_{j}(x,y,\z)$ above)
\begin{align*}
|\partial_{y_{j}}h_{j}(x,y,\z)|\lesssim&
\,\z^{-1}\,(xy)^{\eps}\int\big(\e\b\big)^{1\slash 4}
\,\Pi_{\a+\eps}(ds)\\
&+
\chi_{\{\eps_{j}=1\}}\,\z^{-1\slash 2}\,x^{\eps}y^{\eps-e_{j}}
\int\big(\e\b\big)^{1\slash 2}\,\Pi_{\a+\eps}(ds)\\
&+
\chi_{\{\eps_{j}=1\}}\,(xy)^{\eps-e_{j}}\int\e\b\,\Pi_{\a+\eps}(ds)\\
&+
\chi_{\{\eps_{j}=1\}}\,\z^{-1\slash 2}\,x^{\eps-e_{j}}y^{\eps}
\int\big(\e\b\big)^{1\slash 2}\,\Pi_{\a+\eps}(ds).
\end{align*}
Using the above estimates of $\partial_{y_{i}}h_{j}(x,y,\z)$, the fact that $\t\leq y^{*}$ and Lemma \ref{lemma4.5} twice we get
\begin{align}\label{nier13}
|\partial&_{y_{i}}\delta_{j,x}G_{t}^{\a,\eps}(x,\t)|\\\nonumber
\lesssim&
\sqrt{1-\z^{2}}\,\z^{-d-|\a|-|\eps|-1}\,(xy^{*})^{\eps}
\int\big(\e\fff\big)^{1\slash 64}\,\Pi_{\a+\eps}(ds)\\\nonumber
&+
\,\chi_{\{\eps_{j}=1\}}\,\sqrt{1-\z^{2}}\,\z^{-d-|\a|-|\eps|-1\slash 2}
\,x^{\eps-e_{j}}(y^{*})^{\eps}\int\big(\e\fff\big)^{1\slash 32}\,\Pi_{\a+\eps}(ds)\\\nonumber
&+
\,\chi_{\{\eps_{i}=1\}}\,\sqrt{1-\z^{2}}\,\z^{-d-|\a|-|\eps|-1\slash 2}
\,x^{\eps}(y^{*})^{\eps-e_{i}}\int\big(\e\fff\big)^{1\slash 32}\,\Pi_{\a+\eps}(ds)\\\nonumber
&+
\chi_{\{\eps_{i}=1\}}\chi_{\{\eps_{j}=1\}}
\,\sqrt{1-\z^{2}}\,\z^{-d-|\a|-|\eps|}\,x^{\eps-e_{j}}(y^{*})^{\eps-e_{i}}
\int\big(\e\fff\big)^{1\slash 16}\,\Pi_{\a+\eps}(ds),
\end{align}
provided that $|x-y|>2|y-y'|$. Now Lemma \ref{lemat} (a) (applied with $u=2$ and: $\xi=\rho=0$ or $\xi=e_{j}$, $\rho=0$, or $\xi=0$, $\rho=e_{i}$, or $\xi=e_{j}$, $\rho=e_{i}$) together with Lemma \ref{F2} gives the desired bound.

The proof of the case of $g_{H}^{j,\eps,+}$ in Theorem \ref{kes} is complete.
\end {proof}

\begin {proof}[Proof of Theorem \ref{kes}; the case of $g_{H,*}^{j,\eps,+}$.]
We first show the growth condition. Since $\delta_{j,x}^{*}=-\delta_{j,x}+2x_{j}$, in view of Theorem \ref{kes} (the case of $g_{H}^{j,\eps,+}$) it suffices to show that
$$
\|x_{j}G_{t}^{\a,\eps}(x,y)\|_{L^{2}(dt)}
\lesssim
\frac{1}{w^{+}_{\a}(B(x,|y-x|))},\qquad x\ne y,\qquad j=1,\ldots,d.\\
$$
Taking into account \eqref{hk}, the fact that $x_{j}\leq \sqrt{q_{+}}+\sqrt{q_{-}}$ and Lemma \ref{comp} (b) (specified to $b=1\slash 2$, $A=\z^{-1}$ and $A=\z$) we get
\begin{align}\nonumber
x_{j}G_{t}^{\a,\eps}(x,y)
\lesssim&
\sqrt{1-\z^2}\,\z^{-d-|\a|-|\eps|}\,(xy)^{\eps}\int
\e\b(\sqrt{q_{+}}+\sqrt{q_{-}})\,\Pi_{\a+\eps}(ds)\\\label{nier14}
\lesssim&
\sqrt{1-\z^2}\,\z^{-d-|\a|-|\eps|-1\slash 2}\,(xy)^{\eps}
\int\big(\e\b\big)^{1\slash 2}\,\Pi_{\a+\eps}(ds).
\end{align}
Now an application of Lemma \ref{lemat} (a) (taken with $u=1$ and $\xi=\rho=0$) leads to the required bound.

To prove the smoothness estimates, again in view of the relation $\delta_{j,x}^{*}=-\delta_{j,x}+2x_{j}$ and Theorem \ref{kes} (the case of $g_{H}^{j,\eps,+}$) it suffices to verify that
\begin{align*}
\|x_{j}G_{t}^{\a,\eps}(x,y)-x_{j}^{'}G_{t}^{\a,\eps}(x',y)\|
&\lesssim
\frac{|x-x'|}{|x-y|} \;
\frac{1}{w^{+}_{\a}(B(x,|y-x|))},\qquad |x-y|>2|x-x'|,\\
\|x_{j}G_{t}^{\a,\eps}(x,y)-x_{j}G_{t}^{\a,\eps}(x,y')\|
&\lesssim
\frac{|y-y'|}{|x-y|} \;
\frac{1}{w^{+}_{\a}(B(x,|y-x|))},\qquad |x-y|>2|y-y'|.
\end{align*}
Using the mean value theorem we obtain
\begin{align*}
\big|x_{j}G_{t}^{\a,\eps}(x,y)-x_{j}^{'}G_{t}^{\a,\eps}(x',y)\big|
\leq&
|x-x'|
\big|\nabla_{\!x}\big(x_{j}G_{t}^{\a,\eps}(x,y)\big)\big|_{x=\t}\big|,\\
\big|x_{j}G_{t}^{\a,\eps}(x,y)-x_{j}G_{t}^{\a,\eps}(x,y')\big|
\leq&
|y-y'|
\big|\nabla_{\!y}\big(x_{j}G_{t}^{\a,\eps}(x,y)\big)\big|_{y=\psi}\big|,
\end{align*}
where $\t$, $\psi$ are convex combinations of $x$, $x'$, and $y$, $y'$, respectively, that depend also on $t$. Thus it suffices to show that for any $i,j=1,\ldots,d,$ 
\begin{align*}
\big\|\partial_{x_{i}}\big(x_{j}G_{t}^{\a,\eps}(x,y)\big)\big|_{x=\t}\big\|_{L^{2}(dt)}
&\lesssim
\frac{1}{|x-y|} \;
\frac{1}{w^{+}_{\a}(B(x,|y-x|))},\qquad |x-y|>2|x-x'|,\\
\big\|\partial_{y_{i}}\big(x_{j}G_{t}^{\a,\eps}(x,y)\big)\big|_{y=\psi}\big\|_{L^{2}(dt)}
&\lesssim
\frac{1}{|x-y|} \;
\frac{1}{w^{+}_{\a}(B(x,|y-x|))},\qquad |x-y|>2|y-y'|.
\end{align*}
An elementary computation gives
\begin{align*}
\partial&_{x_{i}}\big(x_{j}G_{t}^{\a,\eps}(x,y)\big)\\
=&
-\frac{1}{2^d}\Big(\frac{1-\z^2}{2\z}\Big)^{d+|\a|+|\eps|}x_{j}(xy)^{\eps}\int\e\b
\Big[\frac{1}{2\z}(x_{i}+y_{i}s_{i})+\frac{\z}{2}(x_{i}-y_{i}s_{i})\Big]
\,\Pi_{\a+\eps}(ds)\\
&+
\big(\chi_{\{\eps_{i}=1\}}+\chi_{\{i=j\}}\big)\,\frac{1}{2^d}\Big(\frac{1-\z^2}{2\z}\Big)^{d+|\a|+|\eps|}
x_{j}\,x^{\eps-e_{i}}y^{\eps}\int\e\b\,\Pi_{\a+\eps}(ds),\\
\partial&_{y_{i}}\big(x_{j}G_{t}^{\a,\eps}(x,y)\big)\\
=&
-\frac{1}{2^d}\Big(\frac{1-\z^2}{2\z}\Big)^{d+|\a|+|\eps|}x_{j}(xy)^{\eps}\int\e\b
\Big[\frac{1}{2\z}(y_{i}+x_{i}s_{i})+\frac{\z}{2}(y_{i}-x_{i}s_{i})\Big]
\,\Pi_{\a+\eps}(ds)\\
&+
\,\chi_{\{\eps_{i}=1\}}\,\frac{1}{2^d}\Big(\frac{1-\z^2}{2\z}\Big)^{d+|\a|+|\eps|}
x_{j}x^{\eps}y^{\eps-e_{i}}\int\e\b\,\Pi_{\a+\eps}(ds).
\end{align*}
Applying the inequality $x_{j}\leq \sqrt{q_{+}}+\sqrt{q_{-}}$ and Lemma \ref{comp} (a), (b) (with $b=1\slash 2$) we get
\begin{align}\label{nier16}
\big|\partial_{x_{i}}\big(x_{j}G_{t}^{\a,\eps}(x,y)\big)\big|
\lesssim&
\sqrt{1-\z^2}\,\z^{-d-|\a|-|\eps|-1}(xy)^{\eps}\int\big(\e\b\big)^{1\slash 4}
\,\Pi_{\a+\eps}(ds)\\\nonumber
&+
\,\chi_{\{\eps_{i}=1\}}\,\sqrt{1-\z^2}\,\z^{-d-|\a|-|\eps|-1\slash 2}x^{\eps-e_{i}}y^{\eps}
\int\big(\e\b\big)^{1\slash 2}\,\Pi_{\a+\eps}(ds),\\\label{nier17}
\big|\partial_{y_{i}}\big(x_{j}G_{t}^{\a,\eps}(x,y)\big)\big|
\lesssim&
\sqrt{1-\z^2}\,\z^{-d-|\a|-|\eps|-1}(xy)^{\eps}\int\big(\e\b\big)^{1\slash 4}
\,\Pi_{\a+\eps}(ds)\\\nonumber
&+
\,\chi_{\{\eps_{i}=1\}}\,\sqrt{1-\z^2}\,\z^{-d-|\a|-|\eps|-1\slash 2}x^{\eps}y^{\eps-e_{i}}
\int\big(\e\b\big)^{1\slash 2}\,\Pi_{\a+\eps}(ds).
\end{align}
Now using the fact that $\t\leq x^{*}$, $\psi\leq y^{*}$, and Lemma \ref{lemma4.5}, we obtain the estimates
\begin{align*}
\big|\partial&_{x_{i}}\big(x_{j}G_{t}^{\a,\eps}(x,y)\big)\big|_{x=\t}\big|\\
\lesssim&
\sqrt{1-\z^2}\,\z^{-d-|\a|-|\eps|-1}(x^{*} y)^{\eps}\int\big(\e\eee\big)^{1\slash 64}
\,\Pi_{\a+\eps}(ds)\\
&+
\,\chi_{\{\eps_{i}=1\}}\,\sqrt{1-\z^2}\,\z^{-d-|\a|-|\eps|-1\slash 2}(x^{*})^{\eps-e_{i}}y^{\eps}
\int\big(\e\eee\big)^{1\slash 32}\,\Pi_{\a+\eps}(ds),\\
\big|\partial&_{y_{i}}\big(x_{j}G_{t}^{\a,\eps}(x,y)\big)\big|_{y=\psi}\big|\\
\lesssim&
\sqrt{1-\z^2}\,\z^{-d-|\a|-|\eps|-1}(xy^{*})^{\eps}\int\big(\e\fff\big)^{1\slash 64}
\,\Pi_{\a+\eps}(ds)\\
&+
\,\chi_{\{\eps_{i}=1\}}\,\sqrt{1-\z^2}\,\z^{-d-|\a|-|\eps|-1\slash 2}x^{\eps}(y^{*})^{\eps-e_{i}}
\int\big(\e\fff\big)^{1\slash 32}\,\Pi_{\a+\eps}(ds),
\end{align*}
provided that $|x-y|>2|x-x'|$ and $|x-y|>2|y-y'|$, respectively.
Finally, combining Lemma \ref{lemat} (a) with Lemma \ref{F2} gives the smoothness conditions.
\end {proof}

\subsection{Lusin's area integrals based on  $\mathbf{\{T_{t}^{\a,\eps,+}\}}$}\label{k4}
$\newline$
In this subsection we show the standard estimates for the kernels 
\begin{align*}
K_{z,t}^{\a,\eps,V}(x,y),\quad 
K_{z,t}^{\a,\eps,H,j}(x,y),\quad
K_{z,t}^{\a,\eps,H,*,j}(x,y),\qquad j=1,\ldots,d,
\end{align*}
valued in the Banach spaces $L^{2}(A,tdtdz)$ (the case of $K_{z,t}^{\a,\eps,V}(x,y)$) or $L^{2}(A,dtdz)$ (the remaining cases), where $A=\{(z,t)\in \RR\times(0,\infty) : |z|<\sqrt{t}\}$.
To achieve this we shall need several additional technical lemmas.

\begin{lem}\label{qz}
Let $x,y\in\R$, $z\in\RR$, $s\in[-1,1]^{d}$. Then
\begin{align*}
q_{\pm}(x+z,y,s)\geq\frac{1}{2}q_{\pm}(x,y,s)-|z|^{2}.
\end{align*}
\end{lem}

\begin {proof}
Since $q_{-}(x,y,s)=q_{+}(x,y,-s)$ we may consider $q_{+}$ only. Moreover, by the structure of $q_{+}$ we may restrict to the one-dimensional case. Then a simple computation shows that
\begin{align*}
q_{+}(x+z,y,s)-\frac{1}{2}q_{+}(x,y,s)+z^{2}=&
\,\frac{1}{2}(x+ys+2z)^{2}+\frac{1}{2}(1-s^{2})y^{2}.
\end{align*}
Since $|s|\leq 1$, the conclusion follows.
\end {proof}

\begin{lem}\label{fakt}
Assume that $\a \in [-1\slash 2, \infty)^d$. Let $x,x'\in\R$ and $z\in\RR$ be such that $x+z\in\R$ and let $\v_{\a}$ be the function given by \eqref{deffi}.
If $\t=\t(x,x',z,\z(t))$ is a convex combination of $x,x'$, then 
\begin{align*}
\int_{|z|<\sqrt{\lo(\z)\slash 2}}\big|\nabla_{\!x}\v_{\a}(x,z,t(\z))\big|_{x=\t}\big|\,\chi_{\{x+z\in\R\}}\chi_{\{x'+z\in\R\}}\,dz
\lesssim
\z^{-1\slash 2}
\end{align*}
uniformly in $x,x',\z$, where $\z$ is related to $t$ as in \eqref{zz}.
\end{lem}

\begin {proof}
It suffices to show that for every $j=1,\ldots,d$, we have
\begin{align*}
\int_{|z|<\sqrt{\lo(\z)\slash 2}}
\big|\partial_{x_{j}}\v_{\a}(x,z,t(\z))\big|_{x=\t}\big|\,\chi_{\{x+z\in\R\}}\chi_{\{x'+z\in\R\}}\,dz
\lesssim
\big(\lo(\z)\big)^{-1\slash 2},
\end{align*}
since $\z\lesssim \lo(\z)$.
An elementary computation gives
\begin{align*}
\partial&_{x_{j}}\v_{\a}(x,z,t)\big|_{x=\t}\\
&=
\frac{(2\a_{j}+1)(\t_{j}+z_{j})^{2\a_{j}}V_{\sqrt{t}}^{\a_{j},+}(\t_{j})-
(\t_{j}+z_{j})^{2\a_{j}+1}\,\partial_{x_{j}}V_{\sqrt{t}}^{\a_{j},+}(x_{j})\big|_{x_j=\t_j}}
{\big(V_{\sqrt{t}}^{\a_{j},+}(\t_{j})\big)^{2}}
\,\prod_{i\ne j}\frac{(\t_{i}+z_{i})^{2\a_{i}+1}}{V_{\sqrt{t}}^{\a_{i},+}(\t_{i})}.
\end{align*}
We estimate this derivative on the set of integration by using the inequality $|z_{j}|\leq \sqrt{\lo(\z)\slash 2}$ and the estimates
\begin{align}\label{os}
V_{\sqrt{t}}^{\a_{j},+}(x_{j})
\simeq 
\sqrt{\lo(\z)}\big(x_{j}+\sqrt{\lo(\z)\slash 2}\big)^{2\a_{j}+1},\qquad
\big|\partial_{x_{j}}V_{\sqrt{t}}^{\a_{j},+}(x_{j})\big|
\leq\big(x_{j}+\sqrt{\lo(\z)\slash 2}\big)^{2\a_{j}+1},
\end{align}
obtaining
\begin{align*}
\big|\partial_{x_{j}}\v_{\a}(x,z,t(\z))\big|_{x=\t}\big|
&\lesssim
\big(\lo(\z)\big)^{-d\slash 2}
\bigg[(2\a_{j}+1)\frac{(\t_{j}+z_{j})^{2\a_{j}}}{\big(\t_{j}+\sqrt{\lo(\z)\slash 2}\big)^{2\a_{j}+1}}
+\big(\lo(\z)\big)^{-1\slash 2}
\bigg]\\
&\equiv I_{1}+I_{2}.
\end{align*}
Now it is not hard to see that the required bound holds for the integral involving $I_{2}$. To estimate the integral related to $I_{1}$, we consider three cases. The case when $\a_{j}=-1\slash 2$ is trivial. When $\a_{j}\in(-1\slash 2,0)$ we observe that the function $s\mapsto \frac{s+z_{j}}{s+\sqrt{\lo(\z)\slash 2}}$ is increasing for $s\geq 0$ and therefore
\begin{align*}
\frac{(\t_{j}+z_{j})^{2\a_{j}}}{\big(\t_{j}+\sqrt{\lo(\z)\slash 2}\big)^{2\a_{j}+1}}
\leq
\frac{((x_{*})_{j}+z_{j})^{2\a_{j}}}{\big((x_{*})_{j}+\sqrt{\lo(\z)\slash 2}\big)^{2\a_{j}+1}},
\end{align*}
where $x_{*}=\min\{x,x'\}$. Using this inequality and observing that $z_{j}>-(x_{*})_j$ if $x_j+z_j>0$ and $x'_{j}+z_j>0$, we obtain
\begin{align*}
\int&_{|z|<\sqrt{\lo(\z)\slash 2}}
\,\,I_{1}\,\chi_{\{x+z\in\R\}}\chi_{\{x'+z\in\R\}}\,dz\\
\lesssim&
\big(\lo(\z)\big)^{-1\slash 2}\int_{-(x_{*})_{j}}^{\sqrt{\lo(\z)\slash 2}}
\frac{((x_{*})_{j}+z_{j})^{2\a_{j}}}{\big((x_{*})_{j}+\sqrt{\lo(\z)\slash 2}\big)^{2\a_{j}+1}}\,dz_{j}
\lesssim
\big(\lo(\z)\big)^{-1\slash 2}.
\end{align*}
Finally, if $\a_{j}\geq 0$ then $(\t_{j}+z_{j})^{2\a_{j}}\leq (\t_{j}+\sqrt{\lo(\z)\slash 2})^{2\a_{j}}$ and
\begin{align*}
I_{1}\lesssim
\big(\lo(\z)\big)^{-(d+1)\slash 2},
\end{align*}
so the conclusion again follows. 
\end {proof}

\begin{lem}\label{lem1}
Assume that $\a\in[-1\slash 2, \infty)^d$ and $\xi, \rho, \eta,\eps\in\Z$ are fixed and such that $\xi+\eta\leq\eps$ and $\rho\leq\eps$. Given $C>0$ and $u\in\mathbb{R}$, consider the function acting on $\R\times\R\times (0,1)$ and defined by
\begin{align*}
p_{u}(x,y,\z)=&
\sqrt{1-\z^2}\,\z^
{-d-|\a|-|\eps|+|\xi|\slash 2+|\rho|\slash 2-u\slash 2}\,
\big(\lo(\z)\big)^{|\eta|\slash 2}x^{\eps-\xi-\eta}y^{\eps-\rho}\exp\bigg(\frac{\lo(\z)}{8\z}\bigg)\\
&\times
\int_{[-1,1]^{d}}\big(\e\b\big)^{C}\,\Pi_{\a+\eps}(ds).
\end{align*}
\begin{itemize}
\item[(a)] If $u\geq 1$, then we have the estimate
\begin{align*}
\|p_{u}(x,y,\z(t))\|_{L^2(dt)}
&\lesssim
\frac{1}{|x-y|^{u-1}} \;
\frac{1}{w^{+}_{\alpha}(B(x,|y-x|))},\qquad x\ne y,
\end{align*}
where $t$ and $\z$ are related as in $\eqref{zz}$.
\item[(b)] If $u\geq 2$, then we also have
\begin{align*}
\|p_{u}(x,y,\z(t))\|_{L^2(tdt)}
&\lesssim
\frac{1}{|x-y|^{u-2}} \;
\frac{1}{w^{+}_{\alpha}(B(x,|y-x|))},\qquad x\ne y.
\end{align*}
\end{itemize}
\end{lem}

\begin {proof}
We will prove only the first inequality, leaving the remaining one to the reader. To show the required estimate we change the variable according to \eqref{zz} and split the region of integration in $\z$ onto $(0,1\slash 2)$ and $(1\slash 2,1)$, denoting the corresponding integrals by $I_{1}$ and $I_{2}$, respectively. Then the conclusion for $I_{1}$ is a straightforward consequence of Lemma \ref{lemat} (a), see the asymptotics \eqref{cal}. We now focus on $I_{2}$. Since $\exp(-s^2)\lesssim\exp(-s)$, when $\z\in(1\slash 2,1)$ we have the estimates $\e\b\lesssim\exp(-\frac{|x|}{4}-\frac{|y|}{4})$, $\exp\big(\frac{\lo(\z)}{8\z}\big)\lesssim(1-\z)^{-1\slash 4}$ and $\z^{-1}\simeq 1$. Thus for $\z\in(1\slash 2,1)$ we obtain
\begin{align*}
I_{2}
&\lesssim
\bigg(\int_{1\slash 2}^{1}
\big(\lo(\z)\big)^{|\eta|}x^{2\eps-2\xi-2\eta}y^{2\eps-2\rho}(1-\z)^{-1\slash 2}
\bigg(\int\exp\Big(-\frac{C|x|}{4}-\frac{C|y|}{4}\Big)\,\Pi_{\a+\eps}(ds)\bigg)^{2}
\,d\z\bigg)^{1\slash 2}\\
&\lesssim
x^{\eps-\xi-\eta}y^{\eps-\rho}\exp\Big(-\frac{C|x|}{4}-\frac{C|y|}{4}\Big)
\lesssim
(|x|+|y|)^{-2d-2|\a|}
\lesssim
\frac{1}{w_{\a}^{+}(B(x,|y-x|))},
\end{align*}
as desired.
\end {proof}

\begin{lem}\label{lem2}
Assume that $\a\in[-1\slash 2, \infty)^d$ and $\xi, \rho, \eps\in\Z$ are fixed and such that $\xi+\eta\leq\eps$ and $\rho\leq\eps$. Given $C>0$ and $u\in\mathbb{R}$, consider the function acting on $\R\times\R\times \big\{(z,\z) : |z|<\sqrt{\lo(\z)\slash 2}\big\}$ and defined by
\begin{align*}
p_{u}(x,y,z,\z)=&
\sqrt{1-\z^2}\,\z^
{-d-|\a|-|\eps|+|\xi|\slash 2+|\rho|\slash 2-u\slash 2}\,
\big(\lo(\z)\big)^{-d\slash 4}(x+z)^{\eps-\xi}y^{\eps-\rho}\,\chi_{\{x+z\in\R\}}\\
&\times
\exp\bigg(\frac{\lo(\z)}{8\z}\bigg)\int_{[-1,1]^{d}}\big(\e\b\big)^{C}\,\Pi_{\a+\eps}(ds).
\end{align*}
\begin{itemize}
\item[(a)] If $u\geq 1$, then we have the estimate
\begin{align*}
\|p_{u}(x,y,z,\z(t))\|_{L^2(A,dtdz)}
&\lesssim
\frac{1}{|x-y|^{u-1}} \;
\frac{1}{w^{+}_{\alpha}(B(x,|y-x|))},\qquad x\ne y,
\end{align*}
where $t$ and $\z$ are related as in $\eqref{zz}$.
\item[(b)] If $u\geq 2$, then we also have
\begin{align*}
\|p_{u}(x,y,z,\z(t))\|_{L^2(A,tdtdz)}
&\lesssim
\frac{1}{|x-y|^{u-2}} \;
\frac{1}{w^{+}_{\alpha}(B(x,|y-x|))},\qquad x\ne y.
\end{align*}
\end{itemize}
\end{lem}

\begin {proof}
As in the proof of Lemma \ref{lem1} we show only the first estimate. Since $|z|<\sqrt{\lo(\z)\slash 2}$ on the set $A$, we get
\begin{equation}\label{New}
(x+z)^{2\eps-2\xi}\chi_{\{x+z\in\R\}}\lesssim
\sum_{0\leq\eta\leq \eps-\xi} x^{2\eps-2\xi-2\eta}\big(\lo(\z)\big)^{|\eta|}.
\end{equation}
Thus we have
$$
\int_{|z|<\sqrt{\lo(\z)\slash 2}} (x+z)^{2\eps-2\xi}\chi_{\{x+z\in\R\}}\,dz
\lesssim
\sum_{0\leq\eta\leq \eps-\xi} x^{2\eps-2\xi-2\eta}\big(\lo(\z)\big)^{|\eta|+d\slash 2}.
$$
Now changing the variable according to \eqref{zz} and then applying the above estimate we obtain
\begin{align*}
\|p&_{u}(x,y,z,\z(t))\|_{L^2(A,dtdz)}\\
=&
\bigg(\int_{0}^{1}\int_{|z|<\sqrt{\lo(\z)\slash 2}}
\Big(\frac{1}{\z}\Big)^{2d+2|\a|+2|\eps|-|\xi|-|\rho|+u}\big(\lo(\z)\big)
^{-d\slash 2}
(x+z)^{2\eps-2\xi}y^{2\eps-2\rho}\chi_{\{x+z\in\R\}}\\
&\times
\exp\bigg(\frac{\lo(\z)}{4\z}\bigg)
\bigg(\int\big(\e\b\big)^{C}\,\Pi_{\a+\eps}(ds)\bigg)^{2}
\,dz\,d\z\bigg)^{1\slash 2}\\
\lesssim&
\sum_{0\leq\eta\leq \eps-\xi}
\bigg(\int_{0}^{1}
\Big(\frac{1}{\z}\Big)^{2d+2|\a|+2|\eps|-|\xi|-|\rho|+u}\big(\lo(\z)\big)^{|\eta|}
x^{2\eps-2\xi-2\eta}y^{2\eps-2\rho}\\
&\times
\exp\bigg(\frac{\lo(\z)}{4\z}\bigg)
\bigg(\int\big(\e\b\big)^{C}\,\Pi_{\a+\eps}(ds)\bigg)^{2}
\,d\z\bigg)^{1\slash 2}.
\end{align*}
This, in view of Lemma \ref{lem1} (a), gives the conclusion.
\end {proof}

\begin {proof}[Proof of Theorem \ref{kes}; the case of $S_{V}^{\eps,+}$.]
Notice that on the set $A\cap\{(z,t) : x+z\in\R\}$ we have, see \eqref{os}, 
\begin{equation}\label{fi}
\v_{\a}(x,z,t)\lesssim \big(\lo(\z)\big)^{-d\slash 2}.
\end{equation}
Using this observation, the estimate \eqref{nier2} of $\partial_{t}G_{t}^{\a,\eps}(x,y)$, Lemma \ref{qz} and the fact that $|z|<\sqrt{\lo(\z)\slash 2}$ on the set $A$, we obtain
\begin{align}\nonumber
|K_{z,t}^{\a,\eps,V}(x,y)|
\lesssim&
\sqrt{1-\z^2}\,\z^{-d-|\a|-|\eps|-1}\,\big(\lo(\z)\big)^{-d\slash 4}(x+z)^{\eps}y^{\eps}\chi_{\{x+z\in\R\}}\\\nonumber
&\times
\int\big(\e\zz\big)^{1\slash 2}\,\Pi_{\a+\eps}(ds)\\\label{nier5}
\lesssim&
\sqrt{1-\z^2}\,\z^{-d-|\a|-|\eps|-1}\,\big(\lo(\z)\big)^{-d\slash 4}(x+z)^{\eps}y^{\eps}\chi_{\{x+z\in\R\}}\\\nonumber
&\times
\exp\bigg(\frac{\lo(\z)}{8\z}\bigg)\int\big(\e\b\big)^{1\slash 4}
\,\Pi_{\a+\eps}(ds).
\end{align}
Now the growth estimate follows with the aid of Lemma \ref{lem2} (b) (specified to $u=2$, $\xi=\rho=0$). 

Next, our task is to show that
\begin{align*}
\big\|K_{z,t}^{\a,\eps,V}(x,y)-K_{z,t}^{\a,\eps,V}(x',y)\big\|_{L^{2}(A,tdtdz)}
\lesssim
\sqrt{\frac{|x-x'|}{|x-y|}} \;
\frac{1}{w_{\a}^{+}(B(x,|y-x|))},\qquad |x-y|>2|x-x'|.
\end{align*}
It is convenient to split the region of integration $A$ above onto four subsets depending on whether $x+z,x'+z$ are in $\R$ or not. More precisely, let 
\begin{align*}
A_{1}&=A\cap\{(z,t) : x+z\in\R, x'+z\in\R\},\\
A_{2}&=A\cap\{(z,t) : x+z\in\R, x'+z\notin\R\},\\
A_{3}&=A\cap\{(z,t) : x+z\notin\R, x'+z\in\R\},\\
A_{4}&=A\cap\{(z,t) : x+z\notin\R, x'+z\notin\R\}.
\end{align*}
We will estimate separately the $L^2(A_i,tdtdz)$ norms, $i=1,\ldots,4$, of the relevant difference.
The treatment of the integral norm over $A_{4}$ is trivial since the integrand vanishes. For the remaining norms we consider three cases.

{\bf Case 1:} \textbf{The norm in} $\mathbf{L^2(A_{1},tdtdz)}.$
Using the triangle inequality we get
\begin{align*}
|K_{z,t}^{\a,\eps,V}(x,y)-K_{z,t}^{\a,\eps,V}(x',y)|
\leq&
\big|\partial_{t}G_{t}^{\a,\eps}(x+z,y)-\partial_{t}G_{t}^{\a,\eps}(x'+z,y)\big|
\sqrt{\v_{\a}(x',z,t)}\\
&+
\big|\partial_{t}G_{t}^{\a,\eps}(x+z,y)\big|\big|\sqrt{\v_{\a}(x,z,t)}-\sqrt{\v_{\a}(x',z,t)}\big|\\
\equiv&
\,I_{1}(x,x',y,z,t)+I_{2}(x,x',y,z,t).
\end{align*}
We will treat $I_{1}$ and $I_{2}$ separately. By the mean value theorem
\begin{align*}
I_{1}(x,x',y,z,t)\leq&
|x-x'|\big|\nabla_{\!x}\partial_{t}G_{t}^{\a,\eps}(x+z,y)\big|_{x=\t}\big|\sqrt{\v_{\a}(x',z,t)},
\end{align*}
where $\t$ is a convex combination of $x$ and $x'$ that depends also on $z$ and $t$. To show the desired bound for the norm of $I_{1}$ it suffices to check that for each $i=1,\ldots,d$, we have 
$$
\big\|\partial_{x_{i}}\big(\partial_{t}G_{t}^{\a,\eps}(x+z,y)\big)\big|_{x=\t}
\sqrt{\v_{\a}(x',z,t)}\big\|_{L^{2}(A_{1},tdtdz)}
\lesssim
\frac{1}{|x-y|} \;
\frac{1}{w_{\a}^{+}(B(x,|y-x|))},
$$
for $|x-y|>2|x-x'|$. Applying \eqref{nier3}, \eqref{fi}, Lemma \ref{qz} and then Lemma \ref{lemma4.5} (with $z=\t$ and then $z=x^{*}$) we get
\begin{align}\nonumber
\big|\partial_{x_{i}}&\big(\partial_{t}G_{t}^{\a,\eps}(x+z,y)\big)\big|_{x=\t}
\sqrt{\v_{\a}(x',z,t)}\big|\\\label{nier7}
\lesssim&
\sqrt{1-\z^2}\,\z^{-d-|\a|-|\eps|-3\slash 2}\,\big(\lo(\z)\big)^{-d\slash 4}
(\t+z)^{\eps}y^{\eps}\int\big(\e\ww\big)^{1\slash 4}\,\Pi_{\a+\eps}(ds)\\\nonumber
&+
\chi_{\{\eps_{i}=1\}}\sqrt{1-\z^{2}}\,\z^{-d-|\a|-|\eps|-1}\,
\big(\lo(\z)\big)^{-d\slash 4}
(\t+z)^{\eps-e_{i}}y^{\eps}\\\nonumber
&\,\,\,\,\,\times
\int\big(\e\ww\big)^{1\slash 2}\,\Pi_{\a+\eps}(ds)\\\nonumber
\lesssim&
\sqrt{1-\z^2}\,\z^{-d-|\a|-|\eps|-3\slash 2}\,\big(\lo(\z)\big)^{-d\slash 4}
(x^{*}+z)^{\eps}y^{\eps}\\\nonumber
&\,\,\,\,\,\times
\exp\bigg(\frac{\lo(\z)}{16\z}\bigg)\int\big(\e\eee\big)^{1\slash 128}\,\Pi_{\a+\eps}(ds)\\\nonumber
&+
\chi_{\{\eps_{i}=1\}}\sqrt{1-\z^{2}}\,\z^{-d-|\a|-|\eps|-1}\,
\big(\lo(\z)\big)^{-d\slash 4}
(x^{*}+z)^{\eps-e_{i}}y^{\eps}\\\nonumber
&\,\,\,\,\,\times
\exp\bigg(\frac{\lo(\z)}{8\z}\bigg)\int\big(\e\eee\big)^{1\slash 64}\,\Pi_{\a+\eps}(ds),
\end{align}
provided that $|x-y|>2|x-x'|$. Now Lemma \ref{lem2} (b) (taken with $u=3$, $\xi=\rho=0$ and $\xi=e_{i}$, $\rho=0$; the application is possible since on $A_{1}$ we have $x^{*}+z\in\R$) together with Lemma \ref{F2} (taken with $z=x^{*}$) leads to the required bound involving $I_{1}$. 

To show the norm estimate of $I_{2}$ we use the inequality $(a-b)^{2}\leq |a^2-b^2|$, which holds for any $a,b\geq 0$. Then the mean value theorem implies
\begin{align}\label{nier10}
\big|\sqrt{\v_{\a}(x,z,t)}-\sqrt{\v_{\a}(x',z,t)}\big|^{2}\leq&
|\v_{\a}(x,z,t)-\v_{\a}(x',z,t)|
\leq
|x-x'|\big|\nabla_{\!x}\v_{\a}(x,z,t)\big|_{x=\t}\big|,
\end{align}
where $\t$ is a convex combination of $x$ and $x'$ depending also on $z$ and $t=t(\z)$. 
Changing the variable according to \eqref{zz} and then applying sequently the above estimate, \eqref{nier2}, Lemma \ref{qz}, inequality \eqref{New} (with $\xi=0$) and Lemma \ref{fakt}, we get
\begin{align*}
\|I&_{2}(x,x',y,z,t)\|_{L^{2}(A_{1},tdtdz)}\\
\lesssim&
\sqrt{|x-x'|}
\bigg(\int_{0}^{1}\int_{|z|<\sqrt{\lo(\z)\slash 2}}
\lo(\z)
\Big(\frac{1}{\z}\Big)^{2d+2|\a|+2|\eps|+2}(x+z)^{2\eps}y^{2\eps}
\chi_{\{x+z\in\R\}}\chi_{\{x'+z\in\R\}}\\
&\times
\big|\nabla_{\!x}\v_{\a}(x,z,t(\z))\big|_{x=\t}\big|\exp\bigg(\frac{\lo(\z)}{4\z}\bigg)
\bigg(\int\big(\e\b\big)^{1\slash 4}\,\Pi_{\a+\eps}(ds)\bigg)^{2}
\,dz\,d\z\bigg)^{1\slash 2}\\
\lesssim&
\sqrt{|x-x'|}
\sum_{0\leq\eta\leq\eps}
\bigg(\int_{0}^{1}
\big(\lo(\z)\big)^{1+|\eta|}
\Big(\frac{1}{\z}\Big)^{2d+2|\a|+2|\eps|+5\slash 2}\,
x^{2\eps-2\eta}y^{2\eps}
\exp\bigg(\frac{\lo(\z)}{4\z}\bigg)\\
&\times
\bigg(\int\big(\e\b\big)^{1\slash 4}\,\Pi_{\a+\eps}(ds)\bigg)^{2}
\,d\z\bigg)^{1\slash 2},
\end{align*}
provided that $|x-y|>2|x-x'|$. Finally, an application of Lemma \ref{lem1} (b) (specified to $u=5\slash 2$, $\xi=\rho=0$) gives the desired estimate, so the conclusion related to $A_{1}$ follows. 

{\bf Case 2:} \textbf{The norm in} $\mathbf{L^2(A_{2},tdtdz)}.$
For $k=1,\ldots,d$, we define the sets 
\begin{align*}
A_{2}^{k}=A\cap\{(z,t) : x+z\in\R, z_{k}\leq-x_{k}'\}. 
\end{align*}
Since these sets cover $A_{2}$ and on each of them $K_{z,t}^{\a,\eps,V}(x',y)=0$, our task reduces to showing that
\begin{align*}
\big\|K_{z,t}^{\a,\eps,V}(x,y)\big\|_{L^{2}(A_{2}^{k},tdtdz)}
&\lesssim
\sqrt{\frac{|x-x'|}{|x-y|}} \;
\frac{1}{w_{\a}^{+}(B(x,|y-x|))},\qquad |x-y|>2|x-x'|.
\end{align*}
Changing the variable according to \eqref{zz}, applying the estimate \eqref{nier5} and then the inequality \eqref{New} (with $\xi=0$), we obtain
\begin{align*}
\big\|K&_{z,t}^{\a,\eps,V}(x,y)\big\|_{L^{2}(A_{2}^{k},tdtdz)}\\
\lesssim&
\sum_{0\leq\eta\leq\eps}
\bigg(\int_{0}^{1}\int_{|z|<\sqrt{\lo(\z)\slash 2}}\chi_{\{-x_{k}<z_{k}\leq-x_{k}'\}}
\big(\lo(\z)\big)^{1+|\eta|-d\slash 2}
\Big(\frac{1}{\z}\Big)^{2d+2|\a|+2|\eps|+2}\,
x^{2\eps-2\eta}y^{2\eps}\\
&\times
\exp\bigg(\frac{\lo(\z)}{4\z}\bigg)
\bigg(\int\big(\e\b\big)^{1\slash 4}\,\Pi_{\a+\eps}(ds)\bigg)^{2}
\,dz\,d\z\bigg)^{1\slash 2}.
\end{align*}
Then using the fact that 
\begin{equation}\label{nier6}
\big(\lo(\z)\big)^{-d\slash 2}\int_{|z|<\sqrt{\lo(\z)\slash 2}}\chi_{\{-x_{k}<z_{k}\leq-x_{k}'\}}\,dz
\lesssim
|x-x'|\big(\lo(\z)\big)^{-1\slash 2}
\lesssim
|x-x'|\z^{-1\slash 2}
\end{equation}
and Lemma \ref{lem1} (b) (taken with $u=5\slash 2$, $\xi=\rho=0$) we arrive at the desired conclusion. 

{\bf Case 3:} \textbf{The norm in} $\mathbf{L^2(A_{3},tdtdz)}.$
Here we proceed in a similar way as in Case 2, this time we also use Lemma \ref{F2} (taken with $\gamma=1\slash 2$ and $z=x'$).

The first smoothness estimate is justified. The proof will be finished once we show that
\begin{align*}
\big\|K_{z,t}^{\a,\eps,V}(x,y)-K_{z,t}^{\a,\eps,V}(x,y')\big\|_{L^{2}(A,tdtdz)}
&\lesssim
\frac{|y-y'|}{|x-y|} \;
\frac{1}{w_{\a}^{+}(B(x,|y-x|))},\qquad |x-y|>2|y-y'|.
\end{align*}
By the mean value theorem it is enough to verify that for any $i=1,\ldots,d$, we have
\begin{align*}
\big\|\partial_{y_{i}}K_{z,t}^{\a,\eps,V}(x,y)\big|_{y=\t}\big\|_{L^{2}(A,tdtdz)}
&\lesssim
\frac{1}{|x-y|} \;
\frac{1}{w_{\a}^{+}(B(x,|y-x|))},\qquad |x-y|>2|y-y'|,
\end{align*}
where $\t$ is a convex combination of $y$ and $y'$ that depends also on $t$ and $z$. 
Taking into account \eqref{kv} and proceeding similarly as in \eqref{nier7}, with the aid of the symmetric version of \eqref{nier3}, \eqref{fi}, Lemma \ref{qz} and Lemma \ref{lemma4.5}, we get
\begin{align*}
\big|\partial_{y_{i}}K_{z,t}^{\a,\eps,V}(x,y)\big|_{y=\t}\big|
\lesssim&
\sqrt{1-\z^2}\,\z^{-d-|\a|-|\eps|-3\slash 2}\,\big(\lo(\z)\big)^{-d\slash 4}
(x+z)^{\eps}(y^{*})^{\eps}\,\chi_{\{x+z\in\R\}}\\
&\,\,\,\,\,\times
\exp\bigg(\frac{\lo(\z)}{16\z}\bigg)\int\big(\e\fff\big)^{1\slash 128}\,\Pi_{\a+\eps}(ds)\\
&+
\chi_{\{\eps_{i}=1\}}\sqrt{1-\z^{2}}\,\z^{-d-|\a|-|\eps|-1}\,
\big(\lo(\z)\big)^{-d\slash 4}
(x+z)^{\eps}(y^{*})^{\eps-e_{i}}\,\chi_{\{x+z\in\R\}}\\
&\,\,\,\,\,\times
\exp\bigg(\frac{\lo(\z)}{8\z}\bigg)\int\big(\e\fff\big)^{1\slash 64}\,\Pi_{\a+\eps}(ds),
\end{align*}
for $|x-y|>2|y-y'|$. Now applications of Lemma \ref{lem2} (b) (taken with $u=3$, $\xi=\rho=0$ and $\xi=0$, $\rho=e_{i}$) and then Lemma \ref{F2} lead to the required bound.

The proof of the case of $S_{V}^{\eps,+}$ in Theorem \ref{kes} is complete.
\end {proof}

\begin {proof}[Proof of Theorem \ref{kes}; the case of $S_{H}^{j,\eps,+}$.]
The reasoning is essentially a repetition of the arguments from the proof of Theorem \ref{kes}, the case of $S_{V}^{\eps,+}$. Firstly, we focus on the growth condition. Using the estimate \eqref{nier8} of $\delta_{j,x}G_{t}^{\a,\eps}(x,y)$ (here we use in addition the inequality $\e\b\leq\big(\e\b\big)^{1\slash 2}$), \eqref{fi} and Lemma \ref{qz}, we get
\begin{align}\label{nier12}
|K_{z,t}^{\a,\eps,H,j}(x,y)|
\lesssim&
\sqrt{1-\z^2}\,\z^{-d-|\a|-|\eps|-1\slash 2}\,\big(\lo(\z)\big)^{-d\slash 4}(x+z)^{\eps}y^{\eps}\chi_{\{x+z\in\R\}}\\\nonumber
&\,\,\,\,\,\times
\exp\bigg(\frac{\lo(\z)}{8\z}\bigg)\int\big(\e\b\big)^{1\slash 4}
\,\Pi_{\a+\eps}(ds)\\\nonumber
&+
\chi_{\{\eps_{j}=1\}}\sqrt{1-\z^2}\,\z^{-d-|\a|-|\eps|}\,\big(\lo(\z)\big)^{-d\slash 4}(x+z)^{\eps-e_{j}}y^{\eps}\chi_{\{x+z\in\R\}}\\\nonumber
&\,\,\,\,\,\times
\exp\bigg(\frac{\lo(\z)}{8\z}\bigg)\int\big(\e\b\big)^{1\slash 4}
\,\Pi_{\a+\eps}(ds),
\end{align}
which in view of Lemma \ref{lem2} (a) (taken with $u=1$, $\xi=\rho=0$ and $\xi=e_{j}$, $\rho=0$) gives the required bound. 

To verify the smoothness conditions we first show that
\begin{align*}
\big\|K_{z,t}^{\a,\eps,H,j}(x,y)-K_{z,t}^{\a,\eps,H,j}(x',y)\big\|_{L^{2}(A,dtdz)}
\lesssim
\sqrt{\frac{|x-x'|}{|x-y|}} \;
\frac{1}{w_{\a}^{+}(B(x,|y-x|))},\qquad |x-y|>2|x-x'|.
\end{align*}
Proceeding similarly as in the proof of the case of $S_{V}^{\eps,+}$ in Theorem \ref{kes}, we split $A$ onto $A_{1},A_{2},A_{3},A_{4}$. The analysis related to $A_{4}$ is trivial. For the remaining sets we consider three cases.

{\bf Case 1:} \textbf{The norm in} $\mathbf{L^2(A_1,dtdz)}.$
On $A_{1}$ we have
\begin{align*}
|K_{z,t}^{\a,\eps,H,j}(x,y)-K_{z,t}^{\a,\eps,H,j}(x',y)|
\leq&
\big|\delta_{j,x}G_{t}^{\a,\eps}(x+z,y)-\delta_{j,x}G_{t}^{\a,\eps}(x'+z,y)\big|
\sqrt{\v_{\a}(x',z,t)}\\
&+
\big|\delta_{j,x}G_{t}^{\a,\eps}(x+z,y)\big|\big|\sqrt{\v_{\a}(x,z,t)}-\sqrt{\v_{\a}(x',z,t)}\big|\\
\equiv&
\,J_{1}(x,x',y,z,t)+J_{2}(x,x',y,z,t).
\end{align*}
We shall treat $J_{1}$ and $J_{2}$ separately. Focusing on $J_{1}$ and using the mean value theorem and \eqref{fi}, we obtain
\begin{align*}
J_{1}(x,x',y,z,t)
\leq&
|x-x'|\big|\nabla_{\!x}\big(\delta_{j,x}G_{t}^{\a,\eps}(x+z,y)\big)\big|_{x=\t}\big|
\big(\lo(\z)\big)^{-d\slash 4},
\end{align*}
where $\t$ is a convex combination of $x$ and $x'$ (notice that $\t$ depends on $z$ and $t$). Thus it suffices to verify that for any $i=1,\ldots,d$,
\begin{equation}\label{est6}
\big\|\partial_{x_{i}}\big(\delta_{j,x}G_{t}^{\a,\eps}(x+z,y)\big)\big|_{x=\t}
\big(\lo(\z)\big)^{-d\slash 4}\big\|_{L^{2}(A_{1},dtdz)}
\lesssim
\frac{1}{|x-y|} \;
\frac{1}{w_{\a}^{+}(B(x,|y-x|))},
\end{equation}
for $|x-y|>2|x-x'|$. Using sequently the estimate of $\partial_{x_{i}}\delta_{j,x}G_{t}^{\a,\eps}(x,y)$ that is implicitly contained in \eqref{nier9} (here we use in addition the inequality $\e\b\leq\big(\e\b\big)^{1\slash 2}$), Lemma \ref{qz} and then Lemma \ref{lemma4.5}, we get
\begin{align*}
\big|\partial_{x_{i}}\big(\delta_{j,x}G_{t}^{\a,\eps}(x+z,y)\big)\big|_{x=\t}\big|
\lesssim&
\sqrt{1-\z^2}\,\z^{-d-|\a|-|\eps|-1}
(x^{*}+z)^{\eps}y^{\eps}
\exp\bigg(\frac{\lo(\z)}{16\z}\bigg)\\
&\,\,\,\,\,\times
\int\big(\e\eee\big)^{1\slash 128}\,\Pi_{\a+\eps}(ds)\\
&+
\chi_{\{\eps_{j}=1\}}\sqrt{1-\z^{2}}\,\z^{-d-|\a|-|\eps|-1\slash 2}
(x^{*}+z)^{\eps-e_{j}}y^{\eps}
\exp\bigg(\frac{\lo(\z)}{8\z}\bigg)\\
&\,\,\,\,\,\times
\int\big(\e\eee\big)^{1\slash 64}\,\Pi_{\a+\eps}(ds)\\
&+
\chi_{\{\eps_{i}=1\}}\sqrt{1-\z^{2}}\,\z^{-d-|\a|-|\eps|-1\slash 2}
(x^{*}+z)^{\eps-e_{i}}y^{\eps}
\exp\bigg(\frac{\lo(\z)}{8\z}\bigg)\\
&\,\,\,\,\,\times
\int\big(\e\eee\big)^{1\slash 64}\,\Pi_{\a+\eps}(ds)\\
&+
\chi_{\{i\ne j\}}\chi_{\{\eps_{i}=1\}}\chi_{\{\eps_{j}=1\}}
\sqrt{1-\z^{2}}\,\z^{-d-|\a|-|\eps|}
(x^{*}+z)^{\eps-e_{i}-e_{j}}y^{\eps}\\
&\,\,\,\,\,\times
\exp\bigg(\frac{\lo(\z)}{8\z}\bigg)\int\big(\e\eee\big)^{1\slash 64}\,\Pi_{\a+\eps}(ds),
\end{align*}
provided that $|x-y|>2|x-x'|$. Finally, applications of Lemma \ref{lem2} (a) (notice that $x^{*}+z\in\R$ on $A_{1}$) and then Lemma \ref{F2} give \eqref{est6}, and hence also the desired bound for the norm of $J_{1}$.

We now consider $J_{2}$. Changing the variable as in \eqref{zz} and then using sequently \eqref{nier10}, \eqref{nier8}, Lemma \ref{qz}, inequality \eqref{New} twice (with $\xi=0$ and $\xi=e_{j}$) and Lemma \ref{fakt}, we see that
\begin{align*}
\|J&_{2}(x,x',y,z,t)\|_{L^{2}(A_{1},dtdz)}\\
\lesssim&
\sqrt{|x-x'|}
\sum_{0\leq\eta\leq\eps}
\bigg(\int_{0}^{1}
\big(\lo(\z)\big)^{|\eta|}
\Big(\frac{1}{\z}\Big)^{2d+2|\a|+2|\eps|+3\slash 2}\,
x^{2\eps-2\eta}y^{2\eps}
\exp\bigg(\frac{\lo(\z)}{4\z}\bigg)\\
&\,\,\,\,\,\times
\bigg(\int\big(\e\b\big)^{1\slash 4}\,\Pi_{\a+\eps}(ds)\bigg)^{2}
\,d\z\bigg)^{1\slash 2}\\
&+
\chi_{\{\eps_{j}=1\}}\sqrt{|x-x'|}
\sum_{0\leq\eta\leq\eps-e_{j}}
\bigg(\int_{0}^{1}
\big(\lo(\z)\big)^{|\eta|}
\Big(\frac{1}{\z}\Big)^{2d+2|\a|+2|\eps|+1\slash 2}\,
x^{2\eps-2e_{j}-2\eta}y^{2\eps}\\
&\,\,\,\,\,\times
\exp\bigg(\frac{\lo(\z)}{4\z}\bigg)
\bigg(\int\big(\e\b\big)^{1\slash 4}\,\Pi_{\a+\eps}(ds)\bigg)^{2}
\,d\z\bigg)^{1\slash 2},
\end{align*}
for $|x-y|>2|x-x'|$. From here the norm estimate for $J_{2}$ follows by Lemma \ref{lem1} (a) (specified to $u=3\slash 2$, $\xi=\rho=0$ and $\xi=e_{j}$, $\rho=0$). This finishes proving the smoothness estimate related to $A_{1}$. 

{\bf Case 2:} \textbf{The norm in} $\mathbf{L^2(A_2,dtdz)}.$
It is enough to check, see the proof of Theorem \ref{kes}, the case of $S_{V}^{\eps,+}$, that for any $k=1,\ldots,d$,
\begin{align*}
\big\|K_{z,t}^{\a,\eps,H,j}(x,y)\big\|_{L^{2}(A_{2}^{k},dtdz)}
&\lesssim
\sqrt{\frac{|x-x'|}{|x-y|}} \;
\frac{1}{w_{\a}^{+}(B(x,|y-x|))},\qquad |x-y|>2|x-x'|.
\end{align*}
Changing the variable as in \eqref{zz}, using the estimate \eqref{nier12} of $K_{z,t}^{\a,\eps,H,j}(x,y)$, inequality \eqref{New} twice and then \eqref{nier6} we obtain
\begin{align*}
\big\|K&_{z,t}^{\a,\eps,H,j}(x,y)\big\|_{L^{2}(A_{2}^{k},dtdz)}\\
\lesssim&
\sqrt{|x-x'|}
\sum_{0\leq\eta\leq\eps}
\bigg(\int_{0}^{1}\big(\lo(\z)\big)^{|\eta|}
\Big(\frac{1}{\z}\Big)^{2d+2|\a|+2|\eps|+3\slash 2}\,
x^{2\eps-2\eta}y^{2\eps}\\
&\,\,\,\,\,\times
\exp\bigg(\frac{\lo(\z)}{4\z}\bigg)
\bigg(\int\big(\e\b\big)^{1\slash 4}\,\Pi_{\a+\eps}(ds)\bigg)^{2}
\,d\z\bigg)^{1\slash 2}\\
&+
\chi_{\{\eps_{j}=1\}}\sqrt{|x-x'|}
\sum_{0\leq\eta\leq\eps-e_{j}}
\bigg(\int_{0}^{1}\big(\lo(\z)\big)^{|\eta|}
\Big(\frac{1}{\z}\Big)^{2d+2|\a|+2|\eps|+1\slash 2}\,
x^{2\eps-2e_{j}-2\eta}y^{2\eps}\\
&\,\,\,\,\,\times
\exp\bigg(\frac{\lo(\z)}{4\z}\bigg)
\bigg(\int\big(\e\b\big)^{1\slash 4}\,\Pi_{\a+\eps}(ds)\bigg)^{2}
\,d\z\bigg)^{1\slash 2},
\end{align*}
which in view of Lemma \ref{lem1} (a) delivers the desired bound.

{\bf Case 3:} \textbf{The norm in} $\mathbf{L^2(A_3,dtdz)}.$
Here the arguments are analogous to those from Case 2. We leave details to the reader.

Eventually, we show the remaining smoothness estimate
\begin{align*}
\big\|K_{z,t}^{\a,\eps,H,j}(x,y)-K_{z,t}^{\a,\eps,H,j}(x,y')\big\|_{L^{2}(A,dtdz)}
&\lesssim
\frac{|y-y'|}{|x-y|} \;
\frac{1}{w_{\a}^{+}(B(x,|y-x|))},\qquad |x-y|>2|y-y'|.
\end{align*}
In view of the mean value theorem it suffices to prove that for any $i=1,\ldots,d$, 
\begin{align*}
\big\|\partial_{y_{i}}K_{z,t}^{\a,\eps,H,j}(x,y)\big|_{y=\t}\big\|_{L^{2}(A,dtdz)}
&\lesssim
\frac{1}{|x-y|} \;
\frac{1}{w_{\a}^{+}(B(x,|y-x|))},\qquad |x-y|>2|y-y'|,
\end{align*}
where $\t$ is a convex combination of $y$ and $y'$. Using the estimate of $\partial_{y_{i}}\delta_{j,x}G_{t}^{\a,\eps}(x,y)$ that is implicitly contained in \eqref{nier13}, together with \eqref{fi}, Lemma \ref{qz} and Lemma \ref{lemma4.5}, we get
\begin{align*}
\big|\partial_{y_{i}}K_{z,t}^{\a,\eps,H,j}(x,y)\big|_{y=\t}\big|
\lesssim&
\sqrt{1-\z^2}\,\z^{-d-|\a|-|\eps|-1}\,\big(\lo(\z)\big)^{-d\slash 4}
(x+z)^{\eps}(y^{*})^{\eps}\,\chi_{\{x+z\in\R\}}\\
&\,\,\,\,\,\times
\exp\bigg(\frac{\lo(\z)}{16\z}\bigg)\int\big(\e\fff\big)^{1\slash 128}\,\Pi_{\a+\eps}(ds)\\
&+
\chi_{\{\eps_{j}=1\}}\sqrt{1-\z^{2}}\,\z^{-d-|\a|-|\eps|-1\slash 2}\,
\big(\lo(\z)\big)^{-d\slash 4}
(x+z)^{\eps-e_{j}}(y^{*})^{\eps}\,\\
&\,\,\,\,\,\times
\chi_{\{x+z\in\R\}}\exp\bigg(\frac{\lo(\z)}{8\z}\bigg)\int\big(\e\fff\big)^{1\slash 64}\,\Pi_{\a+\eps}(ds)\\
&+
\chi_{\{\eps_{i}=1\}}\sqrt{1-\z^{2}}\,\z^{-d-|\a|-|\eps|-1\slash 2}\,
\big(\lo(\z)\big)^{-d\slash 4}
(x+z)^{\eps}(y^{*})^{\eps-e_{i}}\,\\
&\,\,\,\,\,\times
\chi_{\{x+z\in\R\}}\exp\bigg(\frac{\lo(\z)}{8\z}\bigg)\int\big(\e\fff\big)^{1\slash 64}\,\Pi_{\a+\eps}(ds)\\
&+
\chi_{\{\eps_{i}=1\}}\chi_{\{\eps_{j}=1\}}\sqrt{1-\z^{2}}\,\z^{-d-|\a|-|\eps|}\,
\big(\lo(\z)\big)^{-d\slash 4}
(x+z)^{\eps-e_{j}}(y^{*})^{\eps-e_{i}}\,\\
&\,\,\,\,\,\times
\chi_{\{x+z\in\R\}}\exp\bigg(\frac{\lo(\z)}{8\z}\bigg)\int\big(\e\fff\big)^{1\slash 64}\,\Pi_{\a+\eps}(ds),
\end{align*}
provided that $|x-y|>2|y-y'|$. Now combining Lemma \ref{lem2} (a) with Lemma \ref{F2} gives the required estimate. This finishes proving the case of $S_{H}^{j,\eps,+}$ in Theorem \ref{kes}.
\end {proof}

\begin {proof}[Proof of Theorem \ref{kes}; the case of $S_{H,*}^{j,\eps,+}$.]
We first justify the growth condition. Since $\delta_{j,x}^{*}=-\delta_{j,x}+2x_{j}$, in view of the already justified case of $S_{H}^{j,\eps,+}$ in Theorem \ref{kes}, it suffices to verify that
\begin{align*}
\big\|(x_{j}+z_{j})G_{t}^{\a,\eps}(x+z,y)\sqrt{\v_{\a}(x,z,t)}\chi_{\{x+z\in\R\}}\big\|
_{L^{2}(A,dtdz)}
\lesssim
\frac{1}{w_{\a}^{+}(B(x,|y-x|))},\qquad x\ne y.
\end{align*} 
Using the estimate \eqref{nier14} of $x_{j}G_{t}^{\a,\eps}(x,y)$, \eqref{fi} and Lemma \ref{qz}, we see that
\begin{align}\nonumber
(x&_{j}+z_{j})G_{t}^{\a,\eps}(x+z,y)\sqrt{\v_{\a}(x,z,\z)}\chi_{\{x+z\in\R\}}\\\label{nier15}
\lesssim&
\sqrt{1-\z^2}\,\z^{-d-|\a|-|\eps|-1\slash 2}\,\big(\lo(\z)\big)^{-d\slash 4}(x+z)^{\eps}y^{\eps}\chi_{\{x+z\in\R\}}
\exp\bigg(\frac{\lo(\z)}{8\z}\bigg)\\\nonumber
&\times
\int\big(\e\b\big)^{1\slash 4}
\,\Pi_{\a+\eps}(ds).
\end{align} 
Now the growth condition follows with the aid of Lemma \ref{lem2} (a) (specified to $u=1$, $\xi=\rho=0$). 

To prove the first smoothness condition it suffices, in view of the relation $\delta_{j,x}^{*}=-\delta_{j,x}+2x_{j}$ and the already justified case of $S_{H}^{j,\eps,+}$ in Theorem \ref{kes}, to show that
\begin{align*}
\Big\|(a_{j}+z_{j})G_{t}^{\a,\eps}(a+z,y)\sqrt{\v_{\a}(a,z,t)}\chi_{\{a+z\in\R\}}\Big|_{a=x'}^{a=x}
\Big\|
_{L^{2}(A,dtdz)}
\lesssim
\sqrt{\frac{|x-x'|}{|x-y|}} \;
\frac{1}{w_{\a}^{+}(B(x,|y-x|))},
\end{align*} 
for $|x-y|>2|x-x'|$. To do that we split the region of integration onto $A_{1},A_{2},A_{3},A_{4}$, see the proof of the case of $S_{V}^{\eps,+}$ in Theorem \ref{kes}. The analysis related to $A_{4}$ is trivial. Estimates related to the remaining regions are contained in the following three cases. Altogether, they give the desired bound. 

{\bf Case 1:} \textbf{The norm in} $\mathbf{L^2(A_1,dtdz)}.$
Using the triangle inequality we get
\begin{align*}
\Big|(a&_{j}+z_{j})G_{t}^{\a,\eps}(a+z,y)\sqrt{\v_{\a}(a,z,t)}\chi_{\{a+z\in\R\}}\Big|_{a=x'}^{a=x}\Big|\\
\leq&
\big|(x_{j}+z_{j})G_{t}^{\a,\eps}(x+z,y)-(x_{j}'+z_{j})G_{t}^{\a,\eps}(x'+z,y)\big|
\sqrt{\v_{\a}(x',z,t)}\\
&+
\big|(x_{j}+z_{j})G_{t}^{\a,\eps}(x+z,y)\big|\big|\sqrt{\v_{\a}(x,z,t)}-\sqrt{\v_{\a}(x',z,t)}\big|\\
\equiv&
\,L_{1}(x,x',y,z,t)+L_{2}(x,x',y,z,t).
\end{align*} 
First, we analyze $L_{1}$. By the mean value theorem and \eqref{fi} it is enough to check that for any $i,j=1,\ldots,d$, we have
\begin{equation}\label{est7}
\big\|\partial_{x_{i}}\big((x_{j}+z_{j})G_{t}^{\a,\eps}(x+z,y)\big)\big|_{x=\t}
\big(\lo(\z)\big)^{-d\slash 4}\big\|_{L^{2}(A_{1},dtdz)}
\lesssim
\frac{1}{|x-y|} \;
\frac{1}{w_{\a}^{+}(B(x,|y-x|))},
\end{equation}
for $|x-y|>2|x-x'|$, where $\t$ is a convex combination of $x$ and $x'$. Using the inequality \eqref{nier16}, Lemma \ref{qz} and then Lemma \ref{lemma4.5}, we obtain
\begin{align*}
\big|\partial_{x_{i}}\big((x_{j}+z_{j})G_{t}^{\a,\eps}(x+z,y)\big)\big|_{x=\t}\big|
\lesssim&
\sqrt{1-\z^2}\,\z^{-d-|\a|-|\eps|-1}
(x^{*}+z)^{\eps}y^{\eps}
\exp\bigg(\frac{\lo(\z)}{16\z}\bigg)\\
&\,\,\,\,\,\times
\int\big(\e\eee\big)^{1\slash 128}\,\Pi_{\a+\eps}(ds)\\
&+
\chi_{\{\eps_{i}=1\}}\sqrt{1-\z^{2}}\,\z^{-d-|\a|-|\eps|-1\slash 2}
(x^{*}+z)^{\eps-e_{i}}y^{\eps}\\
&\,\,\,\,\,\times
\exp\bigg(\frac{\lo(\z)}{8\z}\bigg)\int\big(\e\eee\big)^{1\slash 64}\,\Pi_{\a+\eps}(ds),
\end{align*} 
provided that $|x-y|>2|x-x'|$. Then combining Lemma \ref{lem2} (a) (notice that on $A_{1}$ we have $x^{*}+z\in\R$) with Lemma \ref{F2} gives \eqref{est7}, and hence also the required bound for the norm of $L_{1}$. 

We now focus on $L_{2}$. Changing the variable according to \eqref{zz} and then applying sequently the estimate \eqref{nier14} of $x_{j}G_{t}^{\a,\eps}(x,y)$, \eqref{nier10}, Lemma \ref{qz}, inequality \eqref{New} (with $\xi=0$) and Lemma \ref{fakt}, we get
\begin{align*}
\|L_{2}(x,x',y,z,t)\|_{L^{2}(A_{1},dtdz)}
\lesssim&
\sqrt{|x-x'|}
\sum_{0\leq\eta\leq\eps}
\bigg(\int_{0}^{1}
\big(\lo(\z)\big)^{|\eta|}
\Big(\frac{1}{\z}\Big)^{2d+2|\a|+2|\eps|+3\slash 2}\,
x^{2\eps-2\eta}y^{2\eps}\\
&\times
\exp\bigg(\frac{\lo(\z)}{4\z}\bigg)
\bigg(\int\big(\e\b\big)^{1\slash 4}\,\Pi_{\a+\eps}(ds)\bigg)^{2}
\,d\z\bigg)^{1\slash 2},
\end{align*}
provided that $|x-y|>2|x-x'|$. From here the conclusion follows with the aid of Lemma \ref{lem1} (a) (taken with $u=3\slash 2$, $\xi=\rho=0$).

{\bf Case 2:} \textbf{The norm in} $\mathbf{L^2(A_2,dtdz)}.$
It suffices to verify that
\begin{align*}
\big\|(x_{j}+z_{j})G_{t}^{\a,\eps}(x+z,y)\sqrt{\v_{\a}(x,z,t)}\big\|
_{L^{2}(A_{2}^{k},dtdz)}
\lesssim
\sqrt{\frac{|x-x'|}{|x-y|}} \;
\frac{1}{w_{\a}^{+}(B(x,|y-x|))},
\end{align*} 
for $|x-y|>2|x-x'|$, where $A_{2}^{k}$ are the sets from the part of the proof of Theorem \ref{kes} concerning $S_{V}^{\eps,+}$. Changing the variable as in \eqref{zz}, using the inequalities \eqref{nier15}, \eqref{New} and then \eqref{nier6}, we obtain
\begin{align*}
\big\|(x_{j}&+z_{j})G_{t}^{\a,\eps}(x+z,y)\sqrt{\v_{\a}(x,z,t)}\big\|_{L^{2}(A_{2}^{k},dtdz)}\\
\lesssim&
\sqrt{|x-x'|}
\sum_{0\leq\eta\leq\eps}
\bigg(\int_{0}^{1}\big(\lo(\z)\big)^{|\eta|}
\Big(\frac{1}{\z}\Big)^{2d+2|\a|+2|\eps|+3\slash 2}\,
x^{2\eps-2\eta}y^{2\eps}\\
&\times
\exp\bigg(\frac{\lo(\z)}{4\z}\bigg)
\bigg(\int\big(\e\b\big)^{1\slash 4}\,\Pi_{\a+\eps}(ds)\bigg)^{2}
\,d\z\bigg)^{1\slash 2}.
\end{align*}
Now an application of Lemma \ref{lem1} (a) (specified to $u=3\slash 2$, $\xi=\rho=0$) leads to the desired bound. 

{\bf Case 3:} \textbf{The norm in} $\mathbf{L^2(A_3,dtdz)}.$
Here the arguments are essentially the same as in Case 2 and thus are omitted.

The proof will be finished once we show the remaining smoothness condition. Again by the relation $\delta_{j,x}^{*}=-\delta_{j,x}+2x_{j}$, the  already justified case of  $S_{H}^{j,\eps,+}$ in Theorem \ref{kes} and the mean value theorem, it suffices to prove that
$$
\big\|\partial_{y_{i}}\big((x_{j}+z_{j})G_{t}^{\a,\eps}(x+z,y)\big)\big|_{y=\t}
\sqrt{\v_{\a}(x,z,t)}\,\chi_{\{x+z\in\R\}}\big\|_{L^{2}(A,dtdz)}
\lesssim
\frac{1}{|x-y|} \;
\frac{1}{w_{\a}^{+}(B(x,|y-x|))},
$$ 
for $|x-y|>2|y-y'|$, where $\t$ is a convex combination of $y$ and $y'$. Using the estimates \eqref{nier17}, \eqref{fi}, Lemma \ref{qz} and Lemma \ref{lemma4.5}, we obtain 
\begin{align*}
\big|\partial&_{y_{i}}\big((x_{j}+z_{j})G_{t}^{\a,\eps}(x+z,y)\big)\big|_{y=\t}\big|
\sqrt{\v_{\a}(x,z,t)}\,\chi_{\{x+z\in\R\}}\\
\lesssim&
\sqrt{1-\z^2}\,\z^{-d-|\a|-|\eps|-1}\,\big(\lo(\z)\big)^{-d\slash 4}
(x+z)^{\eps}(y^{*})^{\eps}\,\chi_{\{x+z\in\R\}}\\
&\,\,\,\,\,\times
\exp\bigg(\frac{\lo(\z)}{16\z}\bigg)\int\big(\e\fff\big)^{1\slash 128}\,\Pi_{\a+\eps}(ds)\\
&+
\chi_{\{\eps_{i}=1\}}\sqrt{1-\z^{2}}\,\z^{-d-|\a|-|\eps|-1\slash 2}\,
\big(\lo(\z)\big)^{-d\slash 4}
(x+z)^{\eps}(y^{*})^{\eps-e_{i}}\,\chi_{\{x+z\in\R\}}\\
&\,\,\,\,\,\times
\exp\bigg(\frac{\lo(\z)}{8\z}\bigg)\int\big(\e\fff\big)^{1\slash 64}\,\Pi_{\a+\eps}(ds),
\end{align*}
provided that $|x-y|>2|y-y'|$. Now the desired bound follows by applying Lemma \ref{lem2} (a) and Lemma \ref{F2}.

The proof of the case of $S_{H,*}^{j,\eps,+}$ in Theorem \ref{kes} is complete. This finishes proving Theorem \ref{kes}.
\end {proof}


\end{document}